\documentclass[12pt,reqno]{article}

\usepackage[usenames]{color}
\usepackage{amssymb}
\usepackage{amsmath}
\usepackage{amsthm}
\usepackage{amsfonts}
\usepackage{amscd}
\usepackage{graphicx}
\usepackage{lscape}
\usepackage[title]{appendix}

\usepackage[colorlinks=true,
linkcolor=webgreen,
filecolor=webbrown,
citecolor=webgreen]{hyperref}

\definecolor{webgreen}{rgb}{0,.5,0}
\definecolor{webbrown}{rgb}{.6,0,0}

\usepackage{color}
\usepackage{fullpage}
\usepackage{float}

\usepackage{graphics}
\usepackage{latexsym}
\usepackage{epsf}
\usepackage{breakurl}

\newcommand{\seqnum}[1]{\href{https://oeis.org/#1}{\rm \underline{#1}}}
\def\suchthat{\, : \, }

\def\modd#1 #2{#1\ \mbox{\rm (mod}\ #2\mbox{\rm )}}
\def\uni{ \, \cup \, }

\def\Enn{\mathbb{N}}

\makeatletter
\def\@fnsymbol#1{\ensuremath{\ifcase#1\or *\or %\dagger\or \ddagger\or
   \mathsection\or \mathparagraph\or \|\or **\or \dagger\dagger
   \or \ddagger\ddagger \else\@ctrerr\fi}}
\makeatother

\begin{document}

\theoremstyle{plain}
\newtheorem{theorem}{Theorem}
\newtheorem{corollary}[theorem]{Corollary}
\newtheorem{lemma}[theorem]{Lemma}
\newtheorem{proposition}[theorem]{Proposition}

\theoremstyle{definition}
\newtheorem{definition}[theorem]{Definition}
\newtheorem{example}[theorem]{Example}
\newtheorem{conjecture}[theorem]{Conjecture}
\newtheorem{openproblem}[theorem]{Open Problem}

\theoremstyle{remark}
\newtheorem{remark}[theorem]{Remark}

\title{Rudin-Shapiro Sums Via Automata Theory and Logic}

\author{
Narad Rampersad\footnote{Research funded by a grant from NSERC, 2019-04111.}\\
Dept. of Mathematics and Statistics\\
University of Winnipeg\\
Winnipeg, MB R2B 2E9\\
Canada\\
\href{mailto:n.rampersad@uwinnipeg.ca}{\tt n.rampersad@uwinnipeg.ca}\\
\and
Jeffrey Shallit\footnote{Research funded by a grant from NSERC, 2018-04118.} \\
School of Computer Science\\
University of Waterloo\\
Waterloo, ON  N2L 3G1\\
Canada\\
\href{mailto:shallit@uwaterloo.ca}{\tt shallit@uwaterloo.ca}\\
}

\maketitle

\begin{abstract}
We show how to obtain, via a unified framework provided by logic and automata theory, many classical results of Brillhart and Morton on Rudin-Shapiro sums.  The techniques also facilitate easy proofs for new results.
\end{abstract}

\section{Introduction}
\label{intro}
The Rudin-Shapiro coefficients
$$(a(n))_{n \geq 0} = (1,1,1,-1,1,1,-1,1,\ldots)$$
form an infinite sequence of $\pm 1$ defined recursively by the identities
\begin{align*}
a(2n) & = a(n) \\
a(2n+1) &= (-1)^n a(n) 
\end{align*}
and the initial condition $a(0) = 1$.
It is sequence
\seqnum{A020985} in the On-Line Encyclopedia of Integer
Sequences (OEIS) \cite{Sloane:2023}.
It was apparently first discovered
by Golay \cite{Golay:1949,Golay:1951}, and later studied
by Shapiro \cite{Shapiro:1951}
and Rudin \cite{Rudin:1959}.  The map $a(n)$ 
can also be defined as
$a(n) = (-1)^{r_n}$,
where $r_n$ counts the number of (possibly overlapping) occurrences of $11$ in the binary representation of $n$
\cite[Satz 1]{Brillhart&Morton:1978}.

The Rudin-Shapiro coefficients have many intriguing properties and have been studied by many authors; for example, see \cite{Allouche:1987,Brillhart&Erdos&Morton:1983,Dekking&MendesFrance&vanderPoorten:1982,MendesFrance&Tenenbaum:1981,MendesFrance:1982}.   They
appear in number theory \cite{Mauduit&Rivat:2015},
analysis \cite{Kahane:1985}, combinatorics \cite{Konieczny:2019},
and even optics \cite{Golay:1949,Golay:1951}, just to name a few
places.

In a classic paper from 1978, written in German, Brillhart and Morton \cite{Brillhart&Morton:1978} studied sums of these coefficients, and defined 
the two sums\footnote{One can make the case that these definitions are ``wrong'', in the sense that many results become significantly simpler to state if the sums are taken over the range $0 \leq i < n$ instead.  But the definitions of Brillhart-Morton are now very well-established, and using a different indexing would also make it harder to compare our results with theirs.}
\begin{align}
s(n) &= \sum_{0\leq i\leq n} a(i) \label{defs} \\
t(n) &= \sum_{0 \leq i\leq n} (-1)^i a(i) .
\label{deft}
\end{align}
The first few values of the functions $s$ and $t$
are given in Table~\ref{tab1}.  They are
sequences \seqnum{A020986} and \seqnum{A020990} respectively,
in the OEIS.  
\begin{table}[H]
\begin{center}
\begin{tabular}{c|ccccccccccccccccccccc}
$n$ & 0& 1& 2& 3& 4& 5& 6& 7& 8& 9&10&11&12&13&14&15&16&17&18&19&20\\
\hline
$s(n)$ & 1& 2& 3& 2& 3& 4& 3& 4& 5& 6& 7& 6& 5& 4& 5& 4& 5& 6& 7& 6& 7 \\
$t(n)$ & 1& 0& 1& 2& 3& 2& 1& 0& 1& 0& 1& 2& 1& 2& 3& 4& 5& 4& 5& 6& 7
\end{tabular}
\end{center}
\caption{First few values of $s(n)$ and $t(n)$.}
\label{tab1}
\end{table}
A priori it is not even clear that these sums are always non-negative, but Brillhart and Morton proved that they are, and also proved many other properties of them.   Most of these properties can be proved by induction, sometimes rather tediously.  

In this paper we show how to replace nearly all of these inductions with techniques from logic and automata theory.  Ultimately, almost all the Brillhart-Morton results\footnote{The main exceptions are the results about
the limit points of $s(n)/\sqrt{n}$ and $t(n)/\sqrt{n}$ in \cite{Brillhart&Morton:1978}.}
can be proved in a simple, unified manner, simply by stating them in first-order logic and applying the {\tt Walnut} theorem-prover \cite{Mousavi:2016,Shallit:2022}.  
In fact, there are basically only three simple inductions in our entire paper.   One is the brief induction used to prove Lemma~\ref{pseudosquare}.  The other two are the inductions used
in Theorem~\ref{thm1} to prove the correctness of our constructed automata, and in these cases the induction step itself can be proved by {\tt Walnut}!   We are also able to easily derive and prove new results; see Section~\ref{new}.

Finally, another justification for this paper is that the same
techniques can easily
be harnessed to handle related sequences; for example, see
\cite{Shallit:2023}.

The paper is organized as follows:  Section~\ref{notation} gives basic notation used in describing base-$b$ expansions.  Section~\ref{automata} gives the automata computing the functions $s(n)$ and $t(n)$ and
justifies their correctness.  In Section~\ref{proofs} we begin proving the results of Brillhart and Morton using our technique, and illustrate the basic ideas, and this continues in Section~\ref{lemmas}. In Section~\ref{special} we show
how to compute various special values of $s$ and $t$.  Then in Section~\ref{inequalities}, we obtain the deepest results of Brillhart and Morton, on inequalities for the Rudin-Shapiro sums.   In Section~\ref{counting} we illustrate how our ideas can be used to enumerate some quantities connected with $s(n)$.    In Section~\ref{new} we use our technique to prove various new results about the Rudin-Shapiro sums.  In Section~\ref{curves} we prove properties about a space-filling curve associated with $s$ and $t$.  Finally, in Section~\ref{further}, we explain how the technique can be used to prove properties of two analogues of the Rudin-Shapiro sums.

We assume the reader is familiar with the basics of automata and
regular expressions, as discussed, for example, in
\cite{Hopcroft&Ullman:1979}.

\section{Notation}
\label{notation}

Let $b\geq 2$ be an integer, and let $\Sigma_b$ denote the alphabet $\{0,1,\ldots, b-1 \}$.  

For a string $x \in \Sigma_b^*$, we let
$[x]_b$ denote the integer
represented by $x$ in base $b$, with the more significant digits at the left.  That is, if $x = c_1 c_2 \cdots c_i$,
then $[x]_b := \sum_{1 \leq j \leq i} c_j b^{i-j}$.  For example,
$[00101011]_2 = [223]_4 = 43$.

For an integer $n \geq 0$, we let $(n)_b$ denote the canonical base-$b$ representation of $n$, starting with the most significant digit, with no leading zeros.  For example,
$(43)_2 = 101011$.  We have
$(0)_b = \epsilon$, the empty
string.

\section{Automata and logic}
\label{automata}

{\tt Walnut} is free software, originally designed by Hamoon Mousavi, that can rigorously prove propositions about automatic sequences  \cite{Mousavi:2016,Shallit:2022}.   If a proposition $P$ has no free variables, one only has to state it in first-order logic and then the {\tt Walnut} prover will prove or disprove it (subject to having enough time and space to complete the calculation)\footnote{In fact, all the {\tt Walnut} code in this paper runs in a matter of milliseconds.}.  If the proposition $P$ has free variables, the program
computes a finite automaton accepting exactly the values of the free variables that make $P$ evaluate to true.

Because of these two features, there is a philosophical choice in using {\tt Walnut}.  Either we can state the desired result as a first-order logical formula and verify it, or, if the theorem involves characterizing a set of numbers $n$ with a certain property, we can simply create a formula with $n$ as a free variable, and consider the resulting automaton as the desired characterization.   In this case, if the automaton is simple enough, we can find a short regular expression specifying the accepted strings, and then interpret it as a function of $n$ in terms of powers of the base $b$. The former is most appropriate when we already know the statement of the theorem we are trying to prove; the latter when we do not yet know the precise characterization that will eventually become a theorem.  In this paper we have used both approaches, to illustrate the ideas.

As is well-known, the Rudin-Shapiro sequence is $2$-automatic and therefore, by a classic theorem of Cobham \cite{Cobham:1972}, also $4$-automatic.   This means there is a deterministic finite automaton with output (DFAO) that on input $n$, expressed in base $4$, reaches a state with output $a(n)$.
This Rudin-Shapiro automaton is
illustrated below in Figure~\ref{rudin4-aut}. 
It is a simple variation on the one given in \cite{Allouche:1987}.
\begin{figure}[H]
\begin{center}
\includegraphics[width=5.5in]{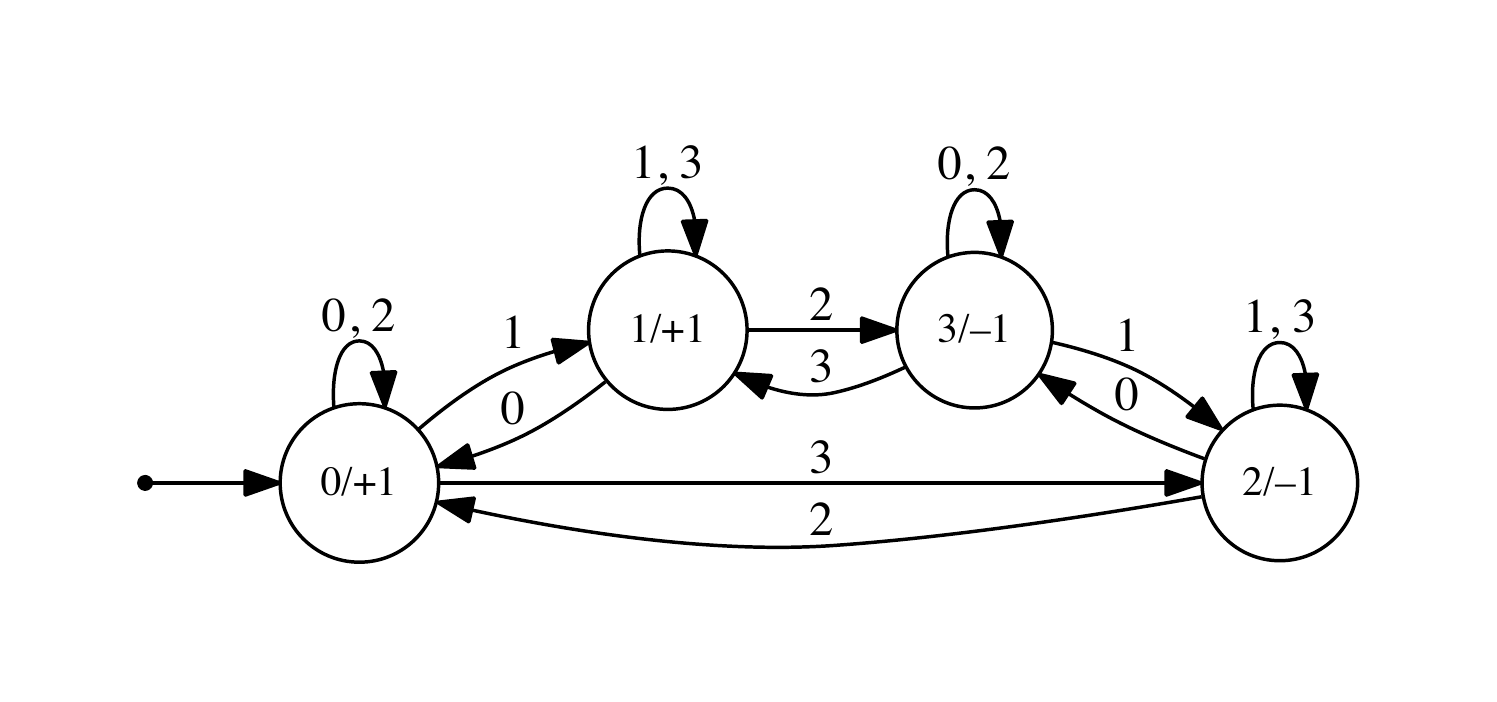}
\end{center}
\vskip -.4in
\caption{DFAO computing the Rudin-Shapiro function, in base $4$.}
\label{rudin4-aut}
\end{figure}
\noindent Here states are labeled $a/b$, where $a$ is the state number and $b$ is the output.
The initial state is state $0$, and
the automaton reads the digits of the base-$4$ representation of $n$, starting with the most significant digit.  Leading zeros are allowed and do not affect the result.

Once the automaton in Figure~\ref{rudin4-aut} is saved
as a file named {\tt RS4.txt},  in {\tt Walnut} we can refer to its value at a variable
$n$ simply by writing {\tt RS4[n]}.
We would like to do the same thing for
the Rudin-Shapiro summatory functions $s(n)$ and $t(n)$ defined in Eqs.~\eqref{defs} and \eqref{deft}, but here we run into a fundamental limitation of {\tt Walnut}:  it can only directly deal with functions of finite range (like automatic sequences).  Since $s(n)$ and $t(n)$ are unbounded, we must find another way to deal with them.

A common way to handle functions in first-order logic is to treat them as {\it relations}:   instead of writing
$f(n) = x$, we construct a relation $R_f(n,x)$ that is true iff $f(n) =x$.
If the relation $R_f(n,x)$ is representable by a deterministic
finite automaton (DFA) taking as input
$n$ in base $b_1$ and $x$ in base $b_2$, in parallel, and accepting iff $R_f(n,x)$
holds, then we say that $f$ is
{\it $(b_1, b_2)$-synchronized}.   For more information about synchronized functions, see \cite{Carpi&Maggi:2001,Shallit:2021h}.

Our first step, then, is to show that the
functions $s(n)$ and $t(n)$ are
$(4,2)$-synchronized.  These automata are illustrated in Figures~\ref{synch1} and \ref{synch2}.
Here accepting states are labeled by
double circles.

\begin{figure}[H]
\begin{center}
\includegraphics[width=5.5in]{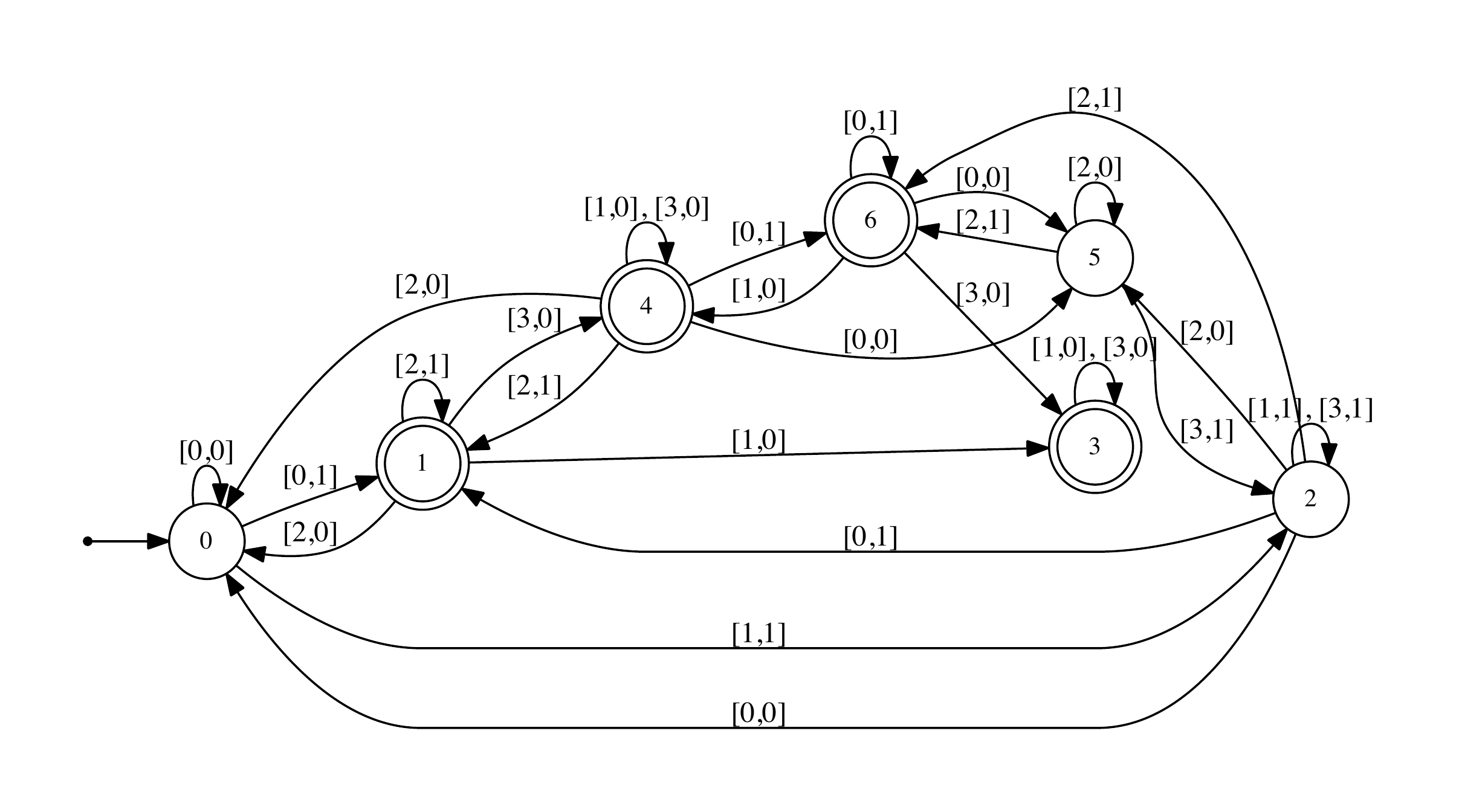}
\end{center}
\caption{Synchronized automaton for $s(n)$.}
\label{synch1}
\end{figure}

\begin{figure}[H]
\begin{center}
\includegraphics[width=5.5in]{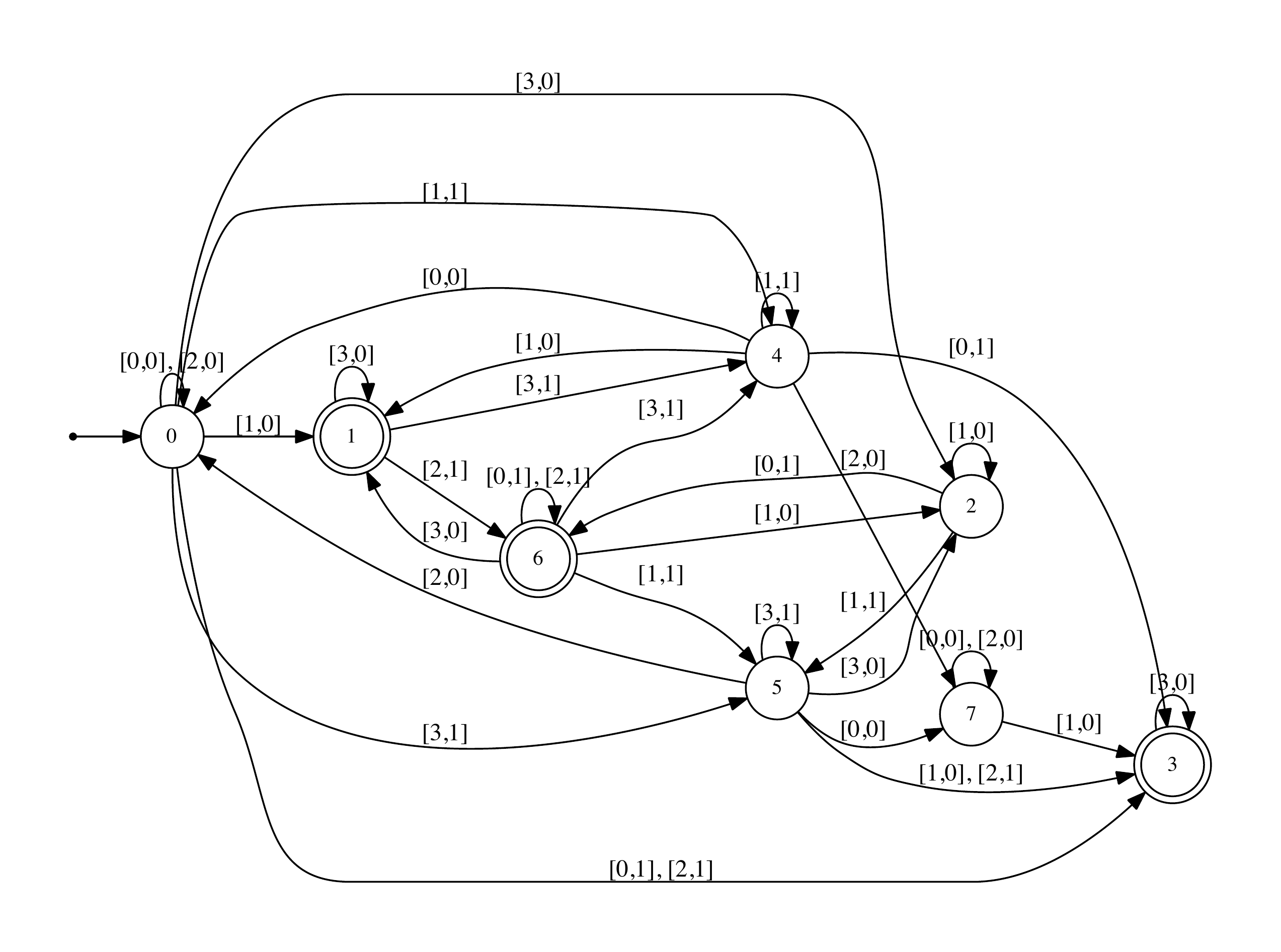}
\end{center}
\caption{Synchronized automaton for $t(n)$.}
\label{synch2}
\end{figure}

We obtained these automata by
``guessing'' them from calculated initial values of the sequences $s$ and $t$, using the
Myhill-Nerode theorem \cite[\S 3.4]{Hopcroft&Ullman:1979}.  However, we will see below in Remark~\ref{remark5} that we could have also deduced them from Satz 3 of \cite{Brillhart&Morton:1978} (Lemma 2
of \cite{Brillhart&Morton:1996}).   The automaton for $s$ is called
{\tt rss} in {\tt Walnut}, and the automaton for $t$ is called
{\tt rst}.

Once we have guessed the automata, we need to verify they are correct. 
\begin{theorem}
The automata in Figs.~\eqref{synch1} and \eqref{synch2} correctly compute $s(n)$ and
$t(n)$.
\label{thm1}
\end{theorem}

\begin{proof}
Let $s_1(n)$ (resp., $t_1(n)$) be the function
computed by the automaton in Fig.~\eqref{synch1} (resp., Fig.~\eqref{synch2}).  We prove that
$s_1 (n) = s(n)$ and $t_1(n) = t(n)$ by
induction on $n$.

First we check
that $s_1(0) = s(0) = 1$ and $t_1(0)=t(0)=1$, which we can see simply by inspecting the automata.

Now assume that $n \geq 1$ and
$s_1(n) = s(n)$ and $t_1(n) = t(n)$.
We prove with
{\tt Walnut\/} that
$s_1(n+1) = s_1(n) + a(n+1)$
and $t_1(n+1) = t_1(n) + (-1)^{n+1} a(n+1)$.
\begin{verbatim}
eval test1 "?msd_4 An,y ($rss(n,y) & RS4[n+1]=@1) => $rss(n+1,?msd_2 y+1)":
eval test2 "?msd_4 An,y ($rss(n,y) & RS4[n+1]=@-1) => $rss(n+1,?msd_2 y-1)":
# show that rss is correct

def even4 "?msd_4 Ek n=2*k":
def odd4 "?msd_4 Ek n=2*k+1":
eval test3 "?msd_4 An,y ($rst(n,y) & ((RS4[n+1]=@1 & $even4(n+1)) | 
   (RS4[n+1]=@-1 & $odd4(n+1)))) => $rst(n+1,?msd_2 y+1)":
eval test4 "?msd_4 An,y ($rst(n,y) & ((RS4[n+1]=@-1 & $even4(n+1)) | 
   (RS4[n+1]=@1 & $odd4(n+1)))) => $rst(n+1, ?msd_2 y-1)":
# show that rst is correct
\end{verbatim}
and {\tt Walnut} returns {\tt TRUE}
for all of these tests.
Now the correctness of our automata follows immediately by induction.
\end{proof}

\begin{remark} Some {\tt Walnut} syntax needs to be explained here.   First, the capital
{\tt A} is {\tt Walnut}'s abbreviation for
$\forall$ (for all); capital {\tt E} is
{\tt Walnut}'s abbreviation for
$\exists$ (there exists); the jargon
{\tt ?msd\_}$b$ for a base $b$ instructs
that a parameter or expression is to be
evaluated using base-$b$ numbers, and
an {\tt @} sign indicates the value of
an automatic sequence (which is allowed to
be negative).  The symbol {\tt \&} is logical {\tt AND}; the symbol {\tt |} is logical {\tt OR}; and the symbol {\tt =>} is logical implication.
\end{remark}

\begin{remark}
There is a small technical wrinkle that we have glossed over, but it needs saying:   the default domain for {\tt Walnut} is $\Enn$, the natural numbers.  But we do not know, a priori, that the functions $s$ and $t$ take only non-negative values.   Therefore, it is conceivable that our verification might fail simply because negative numbers appear as intermediate values in a calculation.  We can check that this does not happen simply by checking that {\tt rss} and {\tt rst} both are truly representations of functions:
\begin{verbatim}
eval test5 "?msd_4 (An Ey $rss(n,y)) & 
   An ~Ex,y ($rss(n,x) & $rss(n,y) & (?msd_2 x!=y))":
eval test6 "?msd_4 (An Ey $rst(n,y)) & 
   An ~Ex,y ($rst(n,x) & $rst(n,y) & (?msd_2 x!=y))":
\end{verbatim}
These commands assert that for every $n$ there
is at least one value $y$ such that
$s(n) = y$, and there are not
two different such $y$, and the same
for $t$.  Both evaluate to {\tt TRUE}.   
This shows that the relations computed by {\tt rss} and {\tt rst} are well-defined functions, and take only non-negative
values.  As a result, we have already deduced
Satz 11 of \cite{Brillhart&Morton:1978}:
$t(n) \geq 0$ for all $n$.
\end{remark}

The advantage of the representation of $s(n)$ and $t(n)$ as synchronized automata is that they essentially encapsulate all the needed knowledge about $s(n)$ and $t(n)$ to replace tedious inductions about them.   The simple induction we used to verify them replaces, in effect, all the other needed inductions.

\begin{remark}
The reader may reasonably ask, as one referee did, why use
base-$4$ for $n$ and base-$2$ for the values of
$s$ and $t$?  The reason is because the values of $s(n)$ and
$t(n)$ grow like $\sqrt{n}$; if we are going to have any hope
of an automaton processing, say, $n$ and $s(n)$ in parallel,
then length considerations show that the base of representation for $n$
must be the square of that for $s(n)$ and $t(n)$.
Furthermore,
by a classical theorem of Cobham \cite{Cobham:1969}, the Rudin-Shapiro
sequence itself can only be generated by an automaton using base
a power of $2$.   This forces the base for $n$ to be $2^{2k}$ for
some $k$ and the base for $s$ and $t$ to be $2^k$.
\end{remark}

All the code necessary to verify the results in this paper can be found at\\
\centerline{\url{https://cs.uwaterloo.ca/~shallit/papers.html} \ .}

\section{Proofs of results}
\label{proofs}

We can now begin to {\it reprove}, and in some cases, {\it improve\/} some of the results of Brillhart and Morton.  Let us start with their Satz 2 \cite{Brillhart&Morton:1978}, reprised as Lemma 1 in 
\cite{Brillhart&Morton:1996}:
\begin{theorem}
We have
\begin{align}
s(2n) &= s(n)+t(n-1), && \quad (n \geq 1); \label{eq6}\\
s(2n+1) &= s(n)+t(n), && \quad (n \geq 0); \label{eq7} \\
t(2n) &= s(n)-t(n-1) , && \quad (n\geq 1); \label{eq8}\\
t(2n+1) &= s(n)-t(n), && \quad (n \geq 0). \label{eq9}
\end{align}
\end{theorem}
\begin{proof}
We use the following {\tt Walnut} commands:
\begin{verbatim}
eval eq3 "?msd_4 An,x,y,z (n>=1 & $rss(2*n,x) & $rss(n,y) & $rst(n-1,z)) 
   => ?msd_2 x=y+z":

eval eq4 "?msd_4 An,x,y,z ($rss(2*n+1,x) & $rss(n,y) & $rst(n,z)) 
   => ?msd_2 x=y+z":

eval eq5 "?msd_4 An,x,y,z (n>=1 & $rst(2*n,x) & $rss(n,y) & $rst(n-1,z)) 
   => ?msd_2 x+z=y":

eval eq6 "?msd_4 An,x,y,z ($rst(2*n+1,x) & $rss(n,y) & $rst(n,z)) 
   => ?msd_2 x+z=y":
\end{verbatim}
and {\tt Walnut} returns {\tt TRUE} for all of them.

For example, the {\tt Walnut} formula {\tt eq3} asserts that for all $n,x,y,z$ if
$n \geq 1$, $s(2n) = x$, $s(n) = y$, 
$t(n-1)=z$, then it must be the case
that $x = y+z$.  It is easily
seen that this is equivalent to the statement of Eq.~\eqref{eq6}.

Note that in our {\tt Walnut} proofs of Eqs.~\eqref{eq7} and \eqref{eq8}, we rearranged the statement to avoid subtractions.  This is because {\tt Walnut}'s basic domain is $\Enn$, the natural numbers, and subtractions that could potentially result in negative numbers might give anomalous results.
\end{proof}

We can now verify Satz 3 of \cite{Brillhart&Morton:1978} (Lemma 2
of \cite{Brillhart&Morton:1996}):
\begin{lemma}
For $n \geq 0$ we have
\begin{align}
s(4n) &= 2s(n)-a(n)\\
s(4n+1) &= s(4n+3) = 2s(n)\\
s(4n+2) &= 2s(n)+(-1)^na(n).
\end{align}
\label{lemma2}
\end{lemma}
\begin{proof}
We use the following {\tt Walnut} commands:
\begin{verbatim}
eval eq7 "?msd_4 An,x,y ($rss(4*n,x) & $rss(n,y)) => 
   ((RS4[n]=@1 => ?msd_2 x+1=2*y) & (RS4[n]=@-1 => ?msd_2 x=2*y+1))":

eval eq8 "?msd_4 An,x,y,z ($rss(4*n+1,x) & $rss(4*n+3,y) & $rss(n,z)) 
   => ?msd_2 x=y & x=2*z":

eval eq9 "?msd_4 An,x,y ($rss(4*n+2,x) & $rss(n,y)) => 
   (((RS4[n]=@1 & $even4(n)) => ?msd_2 x=2*y+1) &
   ((RS4[n]=@-1 & $even4(n)) => ?msd_2 x+1=2*y) &
   ((RS4[n]=@1 & $odd4(n)) => ?msd_2 x+1=2*y) &
   ((RS4[n]=@-1 & $odd4(n)) => ?msd_2 x=2*y+1))":
\end{verbatim}
and {\tt Walnut} returns {\tt TRUE} for all of them.
\end{proof}

\begin{remark}
As it turns out, Lemma~\ref{lemma2} is more or less equivalent to our synchronized automaton depicted in Figure~\ref{synch1}.  To see this, consider a synchronized automaton where the first component represents $n$ in base $4$, while the second component represents $s(n)$ in base $2$,
but using the nonstandard digit set $\{-1, 0, 1\}$ instead of
$\{0, 1\}$.   Reading a bit $i$
in the first component is like
changing the number $n$ read so
far into $4n+i$.   Theorem~\ref{lemma2} says that
$s(4n+i)$ is twice $s(n)$, plus
either $-1, 0, $ or $1$, depending on the value of $a(n)$
and the parity of $n$, both of which are (implicitly) computed
by the base-$4$ DFAO for
$a(n)$.   This gives us the
automaton depicted in
Figure~\ref{t2aut}.
\begin{figure}[H]
\begin{center}
\includegraphics[width=6.5in]{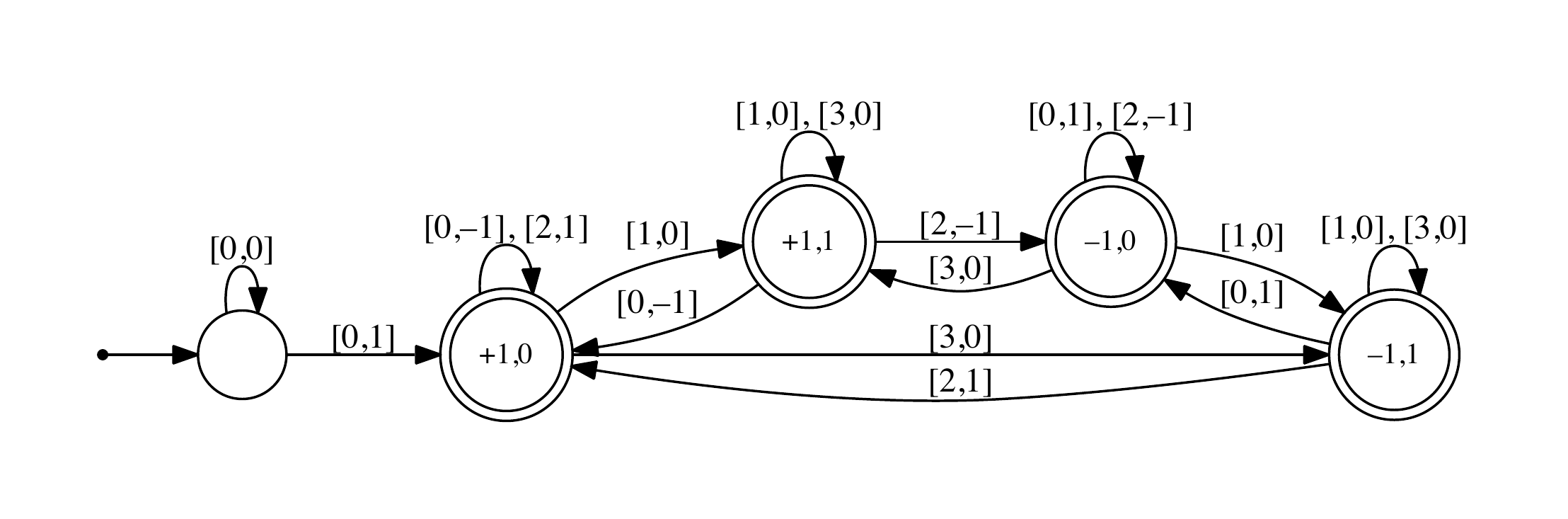}
\end{center}
\caption{Synchronized automaton implementing Theorem~\ref{lemma2}.   The state labels record $a(n)$ and $n \bmod 2$.}
\label{t2aut}
\end{figure}
To get the automaton in Figure~\ref{synch1} from this one, we would need to combine it with a ``normalizer''
that can convert a nonstandard
base-$2$ representation into a standard one.
\label{remark5}
\end{remark}

The sequence $t(n)$ satisfies a similar set of recurrences,
which are given as Satz~4 of \cite{Brillhart&Morton:1978}.

\begin{lemma}
We have
\begin{align}
t(4n) &= 2t(n-1)+a(n), && \quad (n \geq 1);\\
t(4n+1) &= 2t(n-1), && \quad (n \geq 1); \\
t(4n+2) &= t(n)+t(n-1), && \quad (n\geq 1);\\
t(4n+3) &= 2t(n), && \quad (n \geq 0).
\end{align}
\label{satz4}
\end{lemma}
\begin{proof}
We use the following {\tt Walnut} commands:
\begin{verbatim}
eval eq10 "?msd_4 An,x,y (n>=1 & $rst(4*n,x) & $rst(n-1,y)) => 
   ((RS4[n]=@1 => ?msd_2 x=2*y+1) & (RS4[n]=@-1 => ?msd_2 x+1=2*y))":

eval eq11 "?msd_4 An,x,y (n>=1 & $rst(?msd_4 4*n+1,x) &
   $rst(?msd_4 n-1,y)) => ?msd_2 x=2*y":

eval eq12 "?msd_4 An,x,y,z (n>=1 & $rst(4*n+2,x) &
   $rst(n,y) & $rst(n-1,z)) => ?msd_2 x=y+z":

eval eq13 "?msd_4 An,x,y ($rst(4*n+3,x) & $rst(n,y)) => ?msd_2 x=2*y":
\end{verbatim}
and {\tt Walnut} returns {\tt TRUE} for all of them.
\end{proof}

Next we give Theorem~1 of \cite{Brillhart&Morton:1996}.

\begin{theorem}
\leavevmode
\begin{itemize}
\item[(a)] For $k\geq1$, the minimum value of $s(n)$ for $n\in[4^k,4^{k+1}-1]$
is $2^k+1$ and $s(n)$ attains this value only when $n=4^k$ or $n=(5\cdot 4^k-2)/3$.
\item[(b)] For $k\geq0$, the maximum value of $s(n)$ for $n\in[4^k,4^{k+1}-1]$
is $2^{k+2}-1$ and $s(n)$ attains this value only when $n= M_k := \frac23(2^{2k+2}-1)$.
\end{itemize}
\label{thm4}
\end{theorem}
\begin{proof}
We use the following {\tt Walnut} commands:
\begin{verbatim}
reg rss_int msd_4 msd_4 "[0,0]*[1,3][0,3]*":

eval min_rss "?msd_4 n>=4 & $rss(n, x) & Ei,j $rss_int(i,j) &
   i<=n & n<=j & (Ay,m (i<=m & m<=j & $rss(m,y)) =>
   ?msd_2 y>=x)":

eval max_rss "?msd_4 $rss(n, x) & Ei,j $rss_int(i,j) &
   i<=n & n<=j & (Ay,m (i<=m & m<=j & $rss(m,y)) =>
   ?msd_2 y<=x)":
\end{verbatim}
The output of these commands is the automata displayed in Figures~\ref{min_rss} and
\ref{max_rss}.  The automata accept pairs $((n)_4,(s(n))_2)$ where $s(n)$ is extremal for
$n$ in the specified interval.  The first automaton accepts $[0,0]^*[1,1][0,0]^*[0,1]$
and $[0,0]^*[1,1][2,0]^*[2,1]$ and the second automaton accepts $[0,0]^*[0,1][2,1][2,1]^*$.
From this one easily deduces the result.
\end{proof}

\begin{figure}[H]
\begin{center}
\includegraphics[width=5.5in]{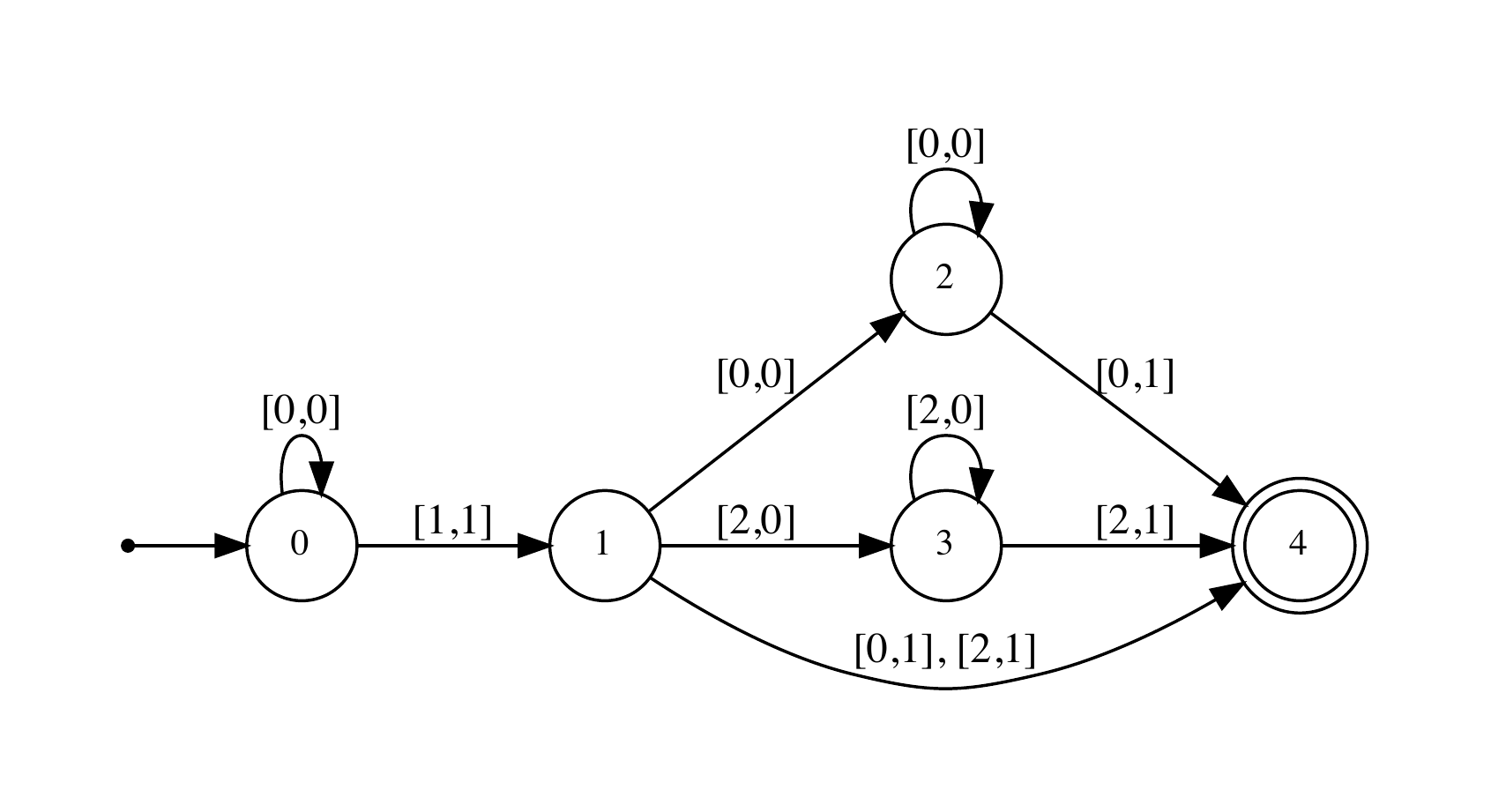}
\end{center}
\caption{Automaton for the minimum value of $s(n)$, $n\in[4^k,4^{k+1}-1]$.}
\label{min_rss}
\end{figure}

\begin{figure}[H]
\begin{center}
\includegraphics[width=5.5in]{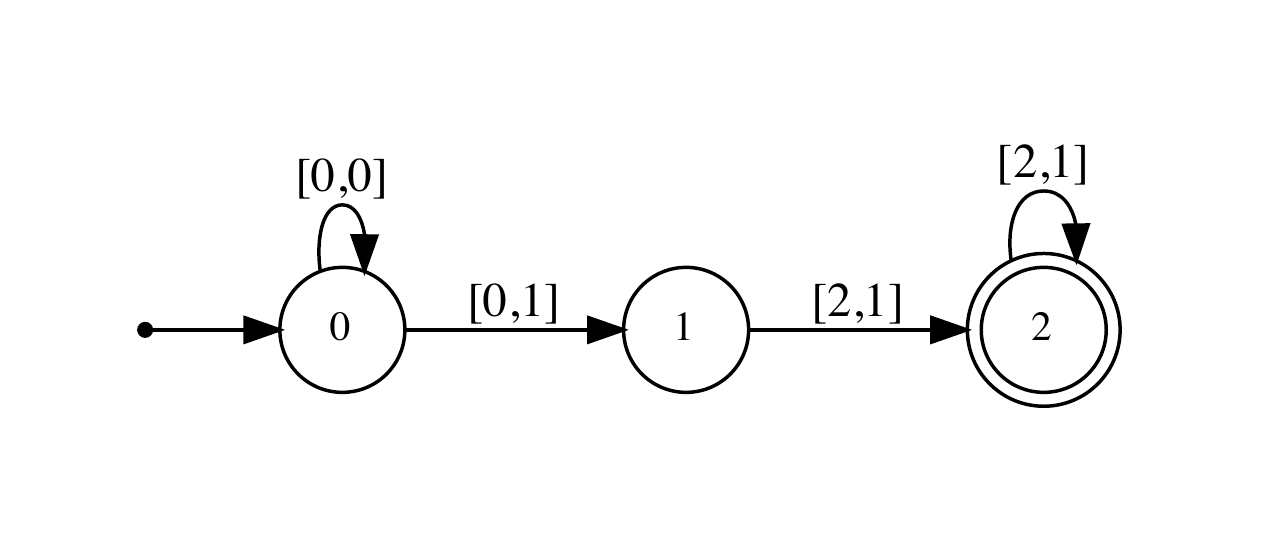}
\end{center}
\caption{Automaton for the maximum value of $s(n)$, $n\in[4^k,4^{k+1}-1]$.}
\label{max_rss}
\end{figure}

\subsection{The $\omega$ function}

Brillhart and Morton defined a function,
$\omega(k)$, as follows:   
$\omega(k)$ is the largest value of
$n$ for which $s(n) = k$.   This is sequence
\seqnum{A020991} in the OEIS.  We can create a $(2,4)$-synchronized 
automaton for
$\omega$ as follows:
\begin{verbatim}
def omega "?msd_4 $rss(n,k) & At (t>n) => ~$rss(t,k)":
\end{verbatim}
We can then create a $(2,4)$-synchronized automaton for
$\omega(n+1)-\omega(n)$ as follows:
\begin{verbatim}
def omegadiff "?msd_4 Et,u $omega((?msd_2 n+1),t) & 
   $omega(n,u) & t=x+u":
\end{verbatim}

Let us now show that
$\omega(s(n)) \leq (10n-2)/3$ for 
$n \geq 2$.  
\begin{verbatim}
def omegas "?msd_4 Ek $rss(n,k) & $omega(k,x)":
eval check_bounds "?msd_4 An,t (n>=2 & $omegas(n,t)) => 3*t+2<=10*n":
\end{verbatim}
Also this bound is optimal, because it
holds for $n = 2^{2i+1}$ and
$i \geq 0$.

Now let us prove Lemma 5 of
Brillhart and Morton \cite{Brillhart&Morton:1996}.
\begin{theorem}
We have 
\begin{align}
\omega(2n) &= 4\omega(n)+3, && \quad n \geq 1;  \label{eq26} \\
\omega(2n+1) &= 4\omega(n+1) + 2,  && \quad n \geq 2, \ 
n+1 \not= 2^r, \ r \geq 2. \label{eq27}
\end{align}
\end{theorem}

\begin{proof}
Let us prove Eq.~\eqref{eq26}.
\begin{verbatim}
eval eq14 "?msd_4 An,x,y ((?msd_2 n>=1) & $omega(n,x) & 
   $omega((?msd_2 2*n), y)) => y=4*x+3":
\end{verbatim}

Brillhart and Morton claim that
the proof of the second equality
is ``much trickier''.  However, with {\tt Walnut} it is not really much more difficult than the previous one.
\begin{verbatim}
reg power2 msd_2 "0*10*":
eval eq15 "?msd_4 An,x,y (?msd_2 n>=2 & (~$power2(?msd_2 n+1)) & 
   $omega((?msd_2 n+1),x) & $omega((?msd_2 2*n+1),y)) => y=4*x+2":
\end{verbatim}
\end{proof}

\section{More lemmas}
\label{lemmas}

Let us now prove Lemma 3
of \cite{Brillhart&Morton:1996}:
\begin{theorem}
We have
\begin{align}
s(n+2^{2k}) &= s(n)+2^k, && \quad 0 \leq n \leq 2^{2k-1} - 1, \ k \geq 1;  \label{eq17}\\
s(n + 2^{2k}) &= -s(n) + 3\cdot 2^k, && \quad 2^{2k-1} \leq n \leq 2^{2k} - 1, \ k \geq 1; \label{eq18} \\
s(n+2^{2k+1}) &= s(n) + 2^{k+1}, && \quad 0 \leq n \leq 2^{2k}-1, \ k \geq 0; \label{eq19} \\
s(n+2^{2k+1}) &= -s(n) + 2^{k+2} , && \quad 2^{2k} \leq n \leq 2^{2k+1} - 1, \ k \geq 0. \label{eq20}
\end{align}
\end{theorem}

\begin{proof}
We can prove these identities with {\tt Walnut}.  One small technical difficulty is that the
equation $x = 2^n$ is not possible to express in the particular first-order logic that {\tt Walnut} is built on; it cannot even multiply arbitrary variables, or raise a number to a power.
Instead, we assert that $x$ is a power of $2$ without exactly specifying {\it which\/} power of $2$ it is.   This brings up a further difficulty, which is that we need to simultaneously express $2^{2k}$ and $2^k$.   Normally this would also not be possible in {\tt Walnut}.  However, in this
case the former is expressed in base $4$ and the latter in base $2$, we can achieve this using the {\tt link42} automaton:
\begin{verbatim}
reg power4 msd_4 "0*10*":
reg link42 msd_4 msd_2 "([0,0]|[1,1])*":
\end{verbatim}
Here {\tt power4} asserts that its argument is a power of $4$; specifically, that its base-$4$ representation looks like $1$ followed by some number of $0$'s, and also allowing any number of leading zeros.  If this is true for $x$,
then $x = 4^k$ for some $k$, and {\tt link42} applied to the pair $(x,y)$ asserts that
$y = 2^k$ (by asserting that the base-$4$
representation $x$ is the same as the base-$2$ representation of
$y$).

To verify Eqs.~\eqref{eq17}--\eqref{eq20}, we use the
following {\tt Walnut} code:
\begin{verbatim}
eval eq16 "?msd_4 An,x,y,z ($power4(x) & x>=4 & 2*n+2<=x & 
   $rss(n,y) & $link42(x,z)) => $rss(n+x,?msd_2 y+z)":
eval eq17 "?msd_4 An,x,y,z ($power4(x) & x>=4 & 2*n>=x & n<x & $rss(n,y)
   & $link42(x,z)) => $rss(n+x,?msd_2 3*z-y)":
eval eq18 "?msd_4 An,x,y,z ($power4(x) & n<x & $rss(n,y) & $link42(x,z))
   => $rss(n+2*x,?msd_2 y+2*z)":
eval eq19 "?msd_4 An,x,y,z ($power4(x) & x<=n & n<2*x & $rss(n,y) &    
   $link42(x,z)) => $rss(n+2*x,?msd_2 4*z-y)":
 \end{verbatim}
 
 \end{proof}

We now turn to
 Lemma 4 in Brillhart and Morton \cite{Brillhart&Morton:1996}.  It is as follows (where we have corrected a typographical error in the original statement).
 \begin{theorem}
 Suppose $n \in [2^{2k}, 2^{2k+1})$.  Then 
 $s(n) \leq 2^{k+1}$, and furthermore, equality holds for $n$ in this range iff
 $n = 2^{2k+1} - 1 - \sum_{0 \leq r < k} e_r 2^{2r+1},$ where the $e_r \in \{0,1\}$.
 \end{theorem}

 \begin{proof}
 We can verify the first claim as follows:
 \begin{verbatim}
 eval lemma4 "?msd_4 An,x,y,z ($power4(x) & x<=n & n<2*x & 
    $link42(x,z) & $rss(n,y)) => ?msd_2 y<=2*z":
 \end{verbatim}
 For the second, let us create a synchronized automaton accepting the base-$4$ representation of $k$ and $n$ for which 
 $s(n) = 2^{k+1}$.
 \begin{verbatim}
 def lemma4a "?msd_4 Ez $power4(x) & x<=n & n<2*x & 
    $link42(x,z) & $rss(n,?msd_2 2*z)":
 \end{verbatim}
 By examining the result, we see that the only accepted paths are labeled with $[1,1]\{[1,0], [3,0]\}^*$.   This is easily seen
 to be the same as the claim
 in the Brillhart-Morton result.
 \end{proof}

\section{Special values}
\label{special}

Along the way, Brillhart and Morton proved a large number of results about special values of the functions $s$ and $t$.  These are very easily proved with
{\tt Walnut}.

Let us start with Examples (``Beispiel'') 5--10
of \cite{Brillhart&Morton:1978}.
\begin{theorem}
We have
\begin{itemize}
\item[(a)] $s(2^k) = 2^{\lfloor (k+1)/2\rfloor} + 1$
for $k \geq 0$;
\item[(b)] $s(2^k - 1) = 2^{\lfloor (k+1)/2\rfloor}$ for $k \geq 0$;
\item[(c)] $s(2^k - 2) = 2^{\lfloor (k+1)/2\rfloor} + (-1)^k$ for $k \geq 1$;
\item[(d)] $s(3\cdot 2^{2k} - 1) = 3\cdot 2^k$ for $k \geq 0$;
\item[(e)] $s(3 \cdot 2^{2k+1} - 1) = 2^{k+2}$ for $k \geq 0$;
\item[(f)] $t(2^{2k}) = 2^k+1$ for $k \geq 1$;
\item[(g)] $t(2^{2k+1}) = 1$ for $k \geq 0$;
\item[(h)] $t(2^{2k}-1) = 2^k$ for $k \geq 0$;
\item[(i)] $t(2^{2k+1} -1) = 0$ for $k \geq 0$;
\item[(j)] $t(2^{2k} - 2) = 2^k - 1$ for $k \geq 1$;
\item[(k)] $t(2^{2k+1} -2) = 1$ for $k \geq 0$;
\item[(l)] $t(3 \cdot 2^k -1) = 2^{\lfloor (k+1)/2\rfloor}$ for $k \geq 0$.
\end{itemize}
\end{theorem}

\begin{proof}
We can verify all of these by straightforward
translation of the assertions into {\tt Walnut}.
First let's write a formula asserting
$x = 2^k$ and $y=2^{\lfloor (k+1)/2\rfloor}$,
where the former is expressed in base $4$ and
the latter in base $2$.
\begin{verbatim}
reg sqrtpow2 msd_4 msd_2 "[0,0]*([1,1]|[0,1][2,0])[0,0]*":
\end{verbatim}
Then we can verify all the assertions as follows:
\begin{verbatim}
reg oddpow2 msd_4 "0*20*":
eval specval_a "?msd_4 Ax,y $sqrtpow2(x,y) => $rss(x,?msd_2 y+1)":
eval specval_b "?msd_4 Ax,y $sqrtpow2(x,y) => $rss(x-1,?msd_2 y)":
eval specval_c1 "?msd_4 Ax,y ($power4(x) & x>1 & $sqrtpow2(x,y))
   => $rss(x-2, ?msd_2 y+1)":
eval specval_c2 "?msd_4 Ax,y ($oddpow2(x) & $sqrtpow2(x,y)) 
   => $rss(x-2, ?msd_2 y-1)":
eval specval_d "?msd_4 Ax,y ($power4(x) & $link42(x,y)) 
   => $rss(3*x-1,?msd_2 3*y)":
eval specval_e "?msd_4 Ax,y ($oddpow2(x) & $sqrtpow2(x,y)) 
   => $rss(3*x-1,?msd_2 2*y)":
eval specval_f "?msd_4 Ax,y ($power4(x) & x>1 & $link42(x,y)) 
   => $rst(x,?msd_2 y+1)":
eval specval_g "?msd_4 Ax $oddpow2(x) => $rst(x,?msd_2 1)":
eval specval_h "?msd_4 Ax,y ($power4(x) & $link42(x,y)) 
   => $rst(x-1,?msd_2 y)":
eval specval_i "?msd_4 Ax $oddpow2(x) => $rst(x-1,?msd_2 0)":
eval specval_j "?msd_4 Ax,y ($power4(x) & x>1 & $link42(x,y)) 
   => $rst(x-2, ?msd_2 y-1)":
eval specval_k "?msd_4 Ax $oddpow2(x) => $rst(x-2,?msd_2 1)":
eval specval_l "?msd_4 Ax,y $sqrtpow2(x,y) => $rst(3*x-1,y)":
\end{verbatim}
\end{proof}

Satz~10 of \cite{Brillhart&Morton:1978} gives the values of $n$
for which $t(n)=0$.  We can find these with the Walnut command
\begin{verbatim}
def satz10 "$rst(?msd_4 n,?msd_2 0)":
\end{verbatim}
The resulting automaton appears in
Figure~\ref{satz10}.
\begin{figure}[H]
\begin{center}
\includegraphics[width=3.5in]{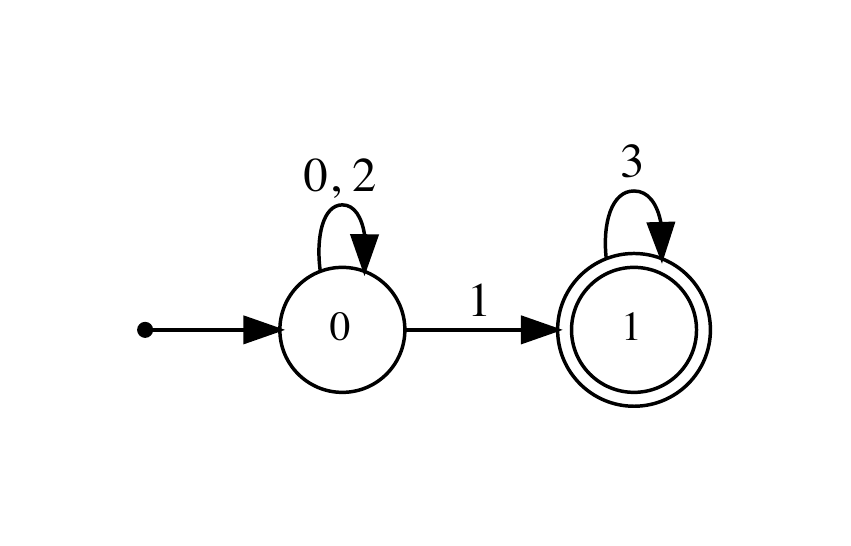}
\end{center}
\vskip -.5in
\caption{DFA accepting those $n$ for which $t(n)=0$, expressed in base $4$.}
\label{satz10}
\end{figure}
Examining this automaton gives the following:
\begin{theorem}
We have $t(n) = 0$ iff 
$(n)_4 \in 13^* \uni 
2\{0,2\}^*13^*$.
\label{nullstellensatz}
\end{theorem}

Next let us
determine those $n$ for which
$s(n) = t(n)$.   We can define the base-$4$ representation of such $n$ with the following {\tt Walnut} command:
\begin{verbatim}
def same "Ex $rss(n,x) & $rst(n,x)":
\end{verbatim}
The resulting automaton appears in
Figure~\ref{same}.
\begin{figure}[H]
\begin{center}
\includegraphics[width=3.0in]{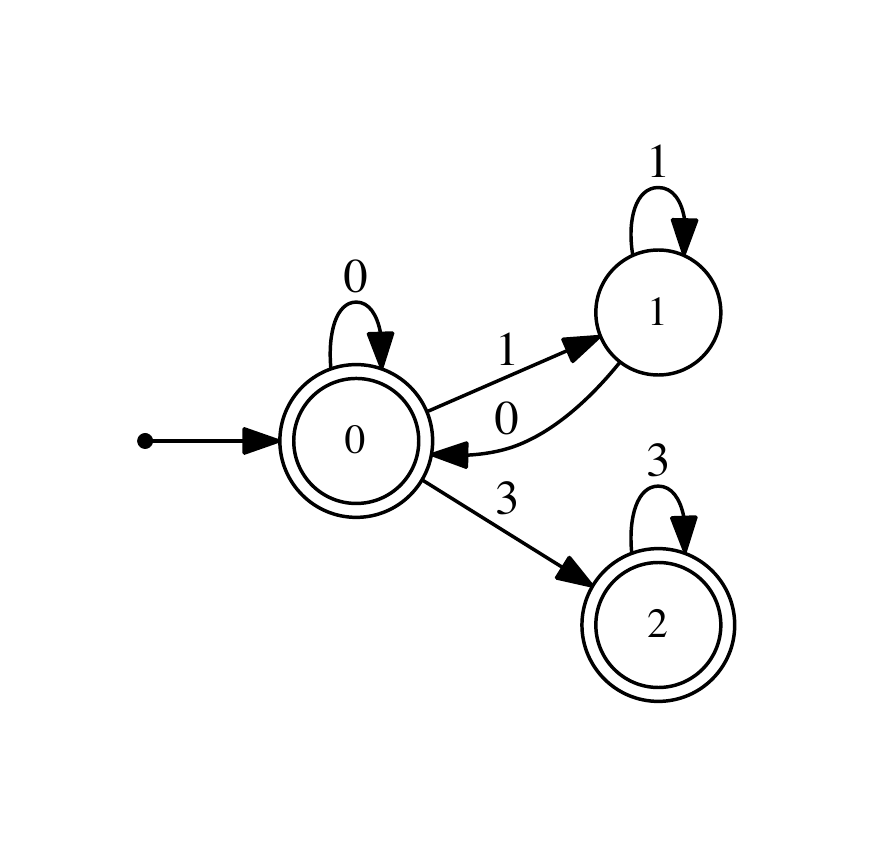}
\end{center}
\vskip -.5in
\caption{DFA accepting those $n$ for which $s(n)=t(n)$, expressed in base $4$.}
\label{same}
\end{figure}
We have therefore proved the following
result, which is given as a Zusatz to Satz~10 of \cite{Brillhart&Morton:1978}.

\begin{theorem}
We have $s(n) = t(n)$ if and only
if $n=0$ or
$(n)_4 \in 3^* \uni 1\{0,1\}^*\, 0 \, 3^*$.
\end{theorem}

From Theorem~\ref{nullstellensatz} we see that the minimum value of $t(n)$ is
$0$ and $t(n)$ takes this value infinitely often.  The next result,
which is an analogue of Theorem~\ref{thm4} (Satz~12 of \cite{Brillhart&Morton:1978}),
gives the maximum value of $t(n)$ on certain intervals.

\begin{theorem}
\leavevmode
\begin{itemize}
\item[(a)] For $k\geq1$, the maximum value of $t(n)$ for $n\in[4^k,2\cdot 4^{k}-1]$
is $2^{k+1}-1$ and $t(n)$ attains this value only when $n=4(2^{2k}-1)/3$.
\item[(b)] For $k\geq1$, the maximum value of $t(n)$ for $n\in[2\cdot 4^k,4^{k+1}-1]$
is $2^{k+1}$ and $t(n)$ attains this value only when $n= 2^{2k+2}-1$.
\end{itemize}
\label{satz12}
\end{theorem}
\begin{proof}
We use the following {\tt Walnut} commands:
\begin{verbatim}
reg rst_int1 msd_4 msd_4 "[0,0]*[1,1][0,3]*":
reg rst_int2 msd_4 msd_4 "[0,0]*[2,3][0,3]*":

eval max_rst1 "?msd_4 n>=2 & $rst(n, ?msd_2 x) & Ei,j $rst_int1(i,j) &
i<=n & n<=j & (Ay,m (i<=m & m<=j & $rst(m,?msd_2 y)) => ?msd_2 y<=x)":

eval max_rst2 "?msd_4 n>=4 & $rst(n, ?msd_2 x) & Ei,j $rst_int2(i,j) &
i<=n & n<=j & (Ay,m (i<=m & m<=j & $rst(m,?msd_2 y)) => ?msd_2 y<=x)":
\end{verbatim}
The output of these commands is the automata displayed in Figures~\ref{min_rst1} and
\ref{max_rst2}.  The automata accept pairs $((n)_4,(t(n))_2)$ where $t(n)$ is extremal for
$n$ in the specified interval.  The first automaton accepts $[0,0]^*[1,1][1,1]^*[0,1]$
and the second automaton accepts $[0,0]^*[3,0][3,0][3,0]^*$.
From this one easily deduces the result.
\end{proof}

\begin{figure}[H]
\begin{center}
\includegraphics[width=5.5in]{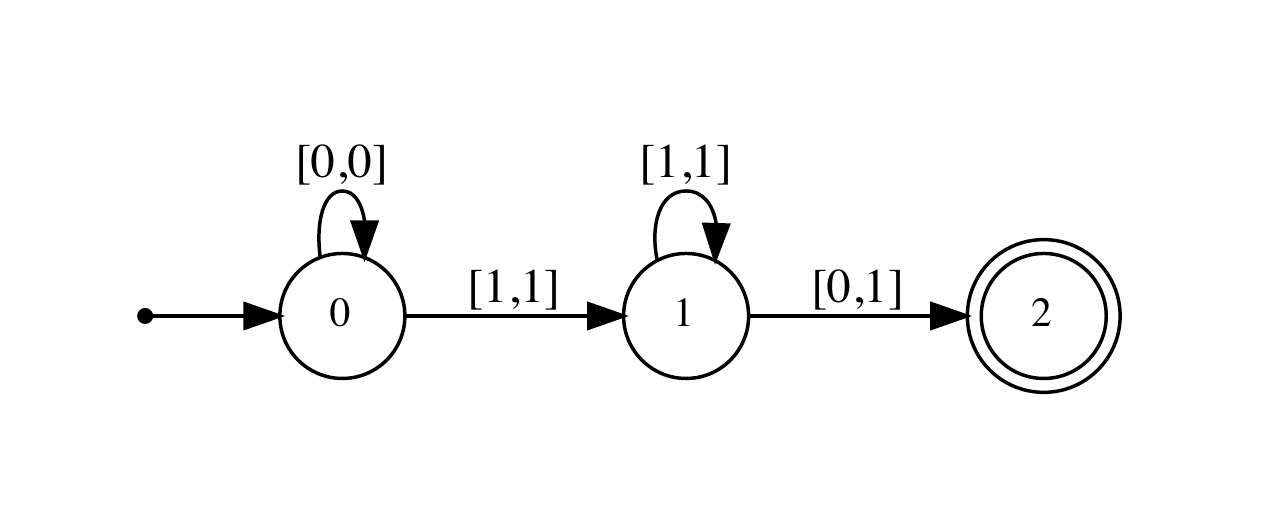}
\end{center}
\caption{Automaton for the maximum value of $t(n)$, $n\in[4^k,2\cdot 4^{k}-1]$.}
\label{min_rst1}
\end{figure}

\begin{figure}[H]
\begin{center}
\includegraphics[width=5.5in]{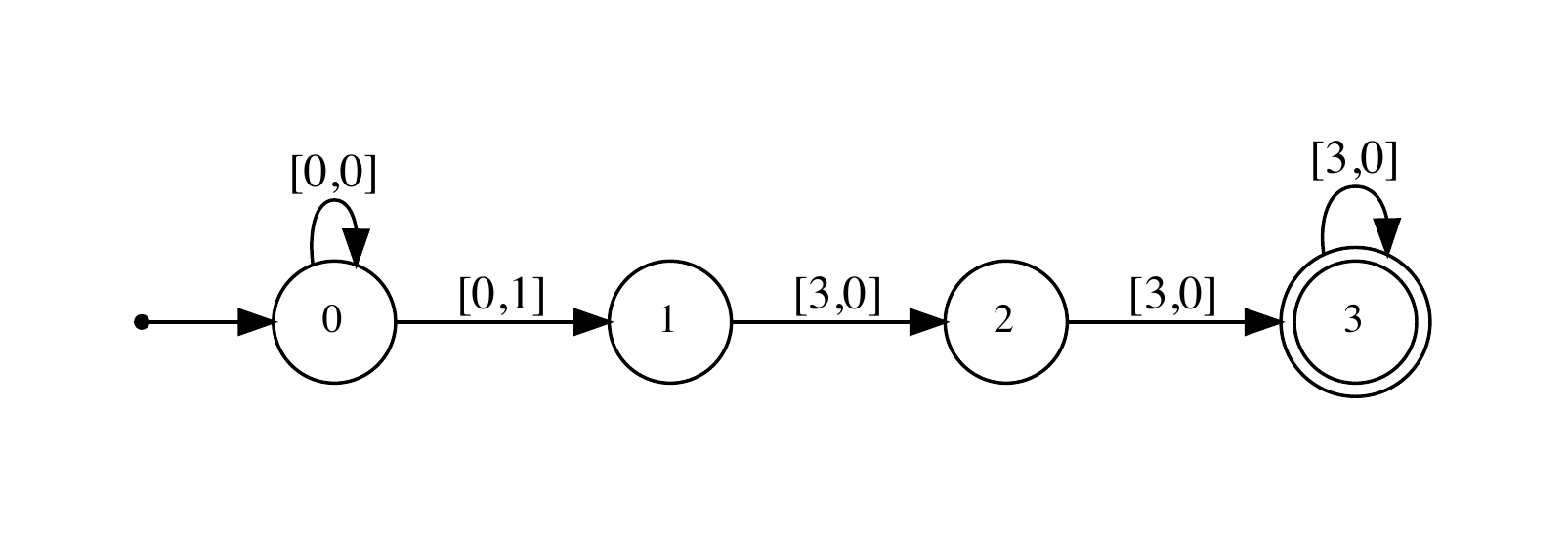}
\end{center}
\caption{Automaton for the maximum value of $t(n)$, $n\in[2\cdot 4^k,4^{k+1}-1]$.}
\label{max_rst2}
\end{figure}

Let us now prove Satz 14 of \cite{Brillhart&Morton:1978}:
\begin{theorem}
We have
\begin{itemize}
\item[(a1)] $s(4(2^{2k}-1)/3) = 2^{k+1} - 1$ for
$k \geq 0$;
\item[(a2)] $s(2^{2k+2} - 1) = 2^{k+1}$ for
$k \geq 0$;
\item[(b1)] $t(2^{2k}) = 2^k + 1$ for $k \geq 1$;
\item[(b2)] $t((5\cdot 2^{2k} - 2)/3) = 2^k - 1$
for $k \geq 0$;
\item[(c)] If $0 \leq s < 2^{k}$ then
$t(2^{2k+1} -1 - 2m(s)) = 2s$;
\item[(d)] If $0 \leq s < 2^{k}$ then
$t(2^{2k+2} - 1 - 2m(s)) = 2^{k+1} - 2s$;
\item[(e)] $t(2(2^{2k+2} -1)/3) = 1$ for $k \geq 0$.
\end{itemize}
\end{theorem}

\begin{proof}
The following straightforward translations of the assertions all evaluate to
{\tt TRUE} in {\tt Walnut}:
\begin{verbatim}
eval eq24a1 "?msd_4 An,x,y,z ($power4(x) & $link42(x,y) & 3*n+4=4*x & 
   $rss(n,z)) => ?msd_2 z+1=2*y":
eval eq24a2 "?msd_4 An,x,y,z ($power4(x) & $link42(x,y) & n+1=4*x & 
   $rss(n,z)) => ?msd_2 z=2*y":
eval eq24b1 "?msd_4 Ax,y,z ($power4(x) & x>1 & $link42(x,y) & $rst(x,z)) 
   => ?msd_2 z=y+1":
eval eq24b2 "?msd_4 An,x,y,z ($power4(x) & $link42(x,y) & 3*n+2=5*x & 
   $rst(n,z)) => ?msd_2 z+1=y":
eval eq24c "?msd_4 An,s,x,y,w,z ($power4(x) & $link42(x,w) & $link42(y,s) 
   & (?msd_2 s<w) & n+2*y+1=2*x & $rst(n,z)) => ?msd_2 z=2*s":
eval eq24d "?msd_4 An,s,x,y,w,z ($power4(x) & $link42(x,w) & $link42(y,s) 
   & (?msd_2 s<w) & n+2*y+1=4*x & $rst(n,z)) => ?msd_2 z+2*s=2*w":
eval eq24e "?msd_4 Ax,n ($power4(x) & 3*n+2=8*x) => $rst(n,?msd_2 1)":
\end{verbatim}
\end{proof}

\section{Inequalities}
\label{inequalities}

We showed in Theorem~\ref{thm1} that $s(n)$ and $t(n)$ are
$(4,2)$-synchronized; furthermore, $s(n)$ and $t(n)$ are both
unbounded as is easily verified with Walnut. A basic result about
synchronized sequences, namely
Theorem~8 of \cite{Shallit:2021h}, immediately implies that
there are constants $c'$ and $c''$ such that
$c' \leq s(n)/\sqrt{n} \leq c''$, and similarly for $t(n)$.  The main
accomplishment of Brillhart and Morton's paper was to determine these
constants.

\begin{theorem}[Brillhart \& Morton]
For $n \geq 1$ we have
\begin{align*}
\sqrt{3n/5} &\leq s(n) \leq \sqrt{6n} \\
0 & \leq t(n) \leq \sqrt{3n} .
\end{align*}
\label{rsbounds}
\end{theorem}
Trying to prove these results by {\it directly\/}  translating the claims into {\tt Walnut} leads to two difficulties:  first, automata cannot compute squares or square roots.
Second, our synchronized automata
work with $n$ expressed in base $4$, but
$s(n)$ and $t(n)$ are expressed in base $2$, and {\tt Walnut} cannot directly compare arbitrary integers expressed in different bases.

However, there is a way around both of these difficulties.
First, we define a kind of ``pseudo-square" function as follows:   $m(n) = 
[(n)_2]_4$.   In other words, $m$ sends $n$ to the integer obtained by interpreting the base-$2$
expansion of $n$ as a number in 
base $4$.  Luckily we have already defined an automaton for $m$ called {\tt link42}; to get an automaton for $m$, we only have to reverse the order of the arguments in {\tt link42}!  

Now we need to see how far away from a real squaring function our pseudo-square function $m(n)$ is.
\begin{lemma}
We have $(n^2+2n)/3 \leq m(n) \leq n^2$.
\label{pseudosquare}
\end{lemma}

\begin{proof}
We can prove the bounds by induction
on $n$.  They are clearly true for
$n = 0$.   Assume $n\geq 1$ and the inequalities hold for
all $n'< n$; we prove them for 
$n$.

Suppose $n$ is even.  Then
$n = 2k$.  Clearly $m(n) = 4m(k)$.
By induction we have
$(k^2+2k)/3 \leq m(k) \leq k^2$,
and multiplying through by $4$ gives
$$ (n^2+2n)/3 = (4k^2+4k)/3 < 4(k^2+2k)/3 \leq 4m(k) \leq 4k^2 = n^2.$$

Suppose $n$ is odd.  Then $n = 2k+1$.
Clearly $m(n) = 4m(k)+1$.  By induction we have
$(k^2+2k)/3 \leq m(k) \leq k^2$.
Multiplying by $4$ and adding $1$ gives
$$(n^2+2n)/3 = 
((2k+1)^2 + 2(2k+1))/3 =
(4k^2 + 8k+ 3)/3 \leq 4m(k)+1  \leq 4k^2+1 \leq (2k+1)^2 = n^2,$$ as desired.
\end{proof}

We can now prove:
\begin{lemma}
For $n \geq 1$ we have
${{3n+7}\over 5} \leq m(s(n)) \leq 3n+1$, and the upper and lower bounds
are tight.
\label{lemma13}
\end{lemma}

\begin{proof}
We use the {\tt Walnut} code
\begin{verbatim}
def maps "?msd_4 Ex $rss(n,x) & $link42(y,x)":
eval ms_lowerbnd "?msd_4 An,y (n>=1 & $maps(n,y)) => y<=3*n+1":
eval ms_upperbnd "?msd_4 An,y (n>=1 & $maps(n,y)) => 3*n+7<=5*y":
\end{verbatim}

To show they are tight, let us show there are infinitely many solutions to $m(s(n)) = 3n+1$ and $m(s(n)) = (3n+7)/5$:
\begin{verbatim}
eval lowerbnd_tight "?msd_4 Am En,y (n>m) & $maps(n,y) & y=3*n+1":
eval upperbnd_tight "?msd_4 Am En,y (n>m) & $maps(n,y) & 5*y=3*n+7":
\end{verbatim}
\end{proof}

As a consequence, we get one lower bound in Theorem~\ref{rsbounds}.
\begin{corollary}
For $n \geq 1$ we have
$$ s(n) \geq \sqrt{{3n+7}\over 5}  .$$
\end{corollary}

\begin{proof}
From Theorem~\ref{pseudosquare} we have
$m(s(n)) \leq s(n)^2$ and from Lemma~\ref{lemma13} we have
${{3n+7}\over 5} \leq m(s(n))$. 
Putting these two bounds together gives
${{3n+7}\over 5} \leq s(n)^2$.
\end{proof}
Note that our lower bound is actually slightly {\it stronger\/} than that of Brillhart-Morton!

\bigskip

To get the upper bound $s(n) \leq \sqrt{6n}$, as in Brillhart-Morton, we need to do more work, since the results we have proved so far only suffice to show that $s(n) \leq \sqrt{9n+3}$.  To get their upper bound, Brillhart and Morton carved the various intervals for $n$ up into three classes and proved the upper bound of $\sqrt{6n}$ for each class.  We'll do the same thing, but use slightly different classes.
By doing so we avoid their complicated induction entirely.

The first class is the easiest:  those $n$ for which $m(s(n))\leq 2n$.   For these $n$, Lemma~\ref{pseudosquare} immediately gives us
$s(n)^2 \leq 6n$, as desired.
Furthermore, the ``exceptional set'' (that is, those $n$ for 
which $m(s(n)) > 2n$) is calculatable with {\tt Walnut}:
\begin{verbatim}
def exceptional_set "?msd_4 Em $maps(n,m) & m>2*n":
\end{verbatim}
The resulting automaton is quite simple (2 states!) and recognizes the
set of base-$4$ expansions
$\{0,2\}^* \uni \{0,2\}^* \ 1 \ \{1,3\}^*$.

We readily see, then, that the
exceptional set consists of
\begin{itemize}
\item[(a)]
numbers whose base-$4$ expansion starts with a $1$ and thereafter consists of $1$'s and $3$'s, and 
\item[(b)] the rest, which must start with a $2$.
\end{itemize}
The numbers in group (a) are easiest to deal with, because
they satisfy the inequality
$M_k/2 \leq n < 2^{2k+1}$
for some $k\geq 0$.  (Recall
that $M_k= (2^{2k+3} - 2)/3$ was defined in
Theorem~\ref{thm4}.)
Now for all $n$ (not just those in the exceptional set) 
in the half-open interval $I_k := [M_k/2,\ 2^{2k+1})$ we can show with
{\tt Walnut} that 
$s(n) \leq 2^{k+1}$, as follows:
\begin{verbatim}
eval maxcheck "?msd_4 An,x,y,z ($power4(x) & 3*n+1>=4*x & n<2*x 
   & $rss(n,y) & $link42(x,z)) => ?msd_2 y<=2*z":
\end{verbatim}
So for all $n \in I_k$ we have
$$ {{s(n)^2} \over {n}} \leq
{ (\max_{n \in I_k} s(n))^2
\over
{\min_{n \in I_k} n}} = 
{{(2^{k+1})^2}\over {M_k/2}} = 3 {{2^{2k+2}}\over{2^{2k+2} - 1}}  
\leq 4.$$
This handles the numbers in group (a).

Finally, we turn to group (b),
which are the hardest to deal with.   These numbers lie in the interval $I'_k = [2^{2k+1}, M_k]$.
We will split these numbers into the following intervals:
$J_{k,i} := [M_k - M_i, M_k - M_{i-1})$
for $0 \leq i < k$.  Since
$M_k - M_{k-1} = 2^{2k+1}$,
the union 
$$J_{k,0} \uni J_{k,1} \uni \cdots \uni J_{k,k-1} \uni \{M_k \}$$ forms a disjoint partition of the interval $I'_k$.  

Now with {\tt Walnut} we can prove that for $n \in J_{k,i}$
we have $s(n) \leq 2^{k+2} -2^{i+1}$.
\begin{verbatim}
eval J_inequality "?msd_4 An,x,y,z,w,m ($rss(n,m) & $power4(x) & $power4(y) 
   & x>y & $link42(x,w) & $link42(y,z) & 8*x<=3*n+8*y & 3*n+2*y<8*x) 
   => ?msd_2 m+2*z<=4*w":
\end{verbatim}

It now follows that
for $n \in J_{k,i}$, $k \geq 1$, and $0 \leq i < k$, we have
$$
{{s(n)^2} \over{n }}
\leq {{(\max_{n\in J_{k,i}} s(n))^2} \over {\min_{n\in J_{k,i}} n}} 
\leq {{(2^{k+2} - 2^{i+1})^2} \over {M_k - M_i}} 
$$ 
and a routine manipulation\footnote{Here are the details.   Since $k \geq 1$ and $0 \leq i < k$ we clearly have
$2^{k+2} > 5\cdot 2^i = 2^{i+2} + 2^i$.  Multiplying by $2^{i+2}$ gives us
$2^{k+i+4} > 2^{2i+4} + 2^{2i+2}$.   Adding $2^{2k+4}$ to both sides, and rearranging gives $2^{2k+4} - 2^{k+i+4} + 2^{2i+2} < 2^{2k+4} - 2^{2i+4}$.  In other words,
$(2^{k+2}-2^{i+1})^2 < 2 (2^{2k+3} - 2^{2i+3})$.
Hence $(2^{k+2} - 2^{i+1})^2/(2^{2k+3} - 2^{2i+3}) < 2$, and so
$(2^{k+2}-2^{i+1})^2/(M_k -M_i) < 6$.}
shows this is less than $6$.

The only remaining case is
$M_k$.  But then $s(M_k) = 2^{k+2} - 1$, and then
$s(M_k)^2 <6 M_k$ by another
routine calculation.

Finally, we should verify that we have really covered all the possible $n$:
\begin{verbatim}
def left_endpoint "?msd_4 3*z+8*y=8*x":
def right_endpoint "?msd_4 3*z+2*y=8*x":
eval check_all "?msd_4 An (n>=1) => ((~$exceptional_set(n)) |
   (Ex $power4(x) & 4*x<=3*n+1 & n<2*x) |
   (Ex,y,z,w $power4(x) & $power4(y) & x>y & 
   $left_endpoint(x,y,z) & $right_endpoint(x,y,w) & n>=z & n<w) | 
   (Ex $power4(x) & 3*n+2=8*x))":
\end{verbatim}
which evaluates to {\tt TRUE}.

Thus we have proved one upper bound from
Theorem~\ref{rsbounds}:
\begin{theorem}
$s(n) \leq \sqrt{6n}$ for $n \geq 1$.
\end{theorem}

Using exactly the same techniques we can prove
\begin{lemma}
For all $n \geq 0$ we
have $m(t(n)) \leq n+1$,
with equality iff 
$(n)_4 \in (0 \uni 11^*0)^*3^*$.
\label{thm17}
\end{lemma}

\begin{proof}
We use the following {\tt Walnut} code:
\begin{verbatim}
def mapt "?msd_4 Ex $rst(n,x) & $link42(y,x)":
eval bnd "?msd_4 An,z $mapt(n,z) => z<=n+1":
def except2 "?msd_4 Ez $mapt(n,z) & z=n+1":
\end{verbatim}
The command {\tt bnd} returns {\tt TRUE}, and the command
{\tt except2} computes a simple automaton of 3 states
accepting the regular expression $(0 \uni 11^*0)^*3^*$. 
\end{proof}

From the first claim of Lemma~\ref{thm17} we see that
$m(t(n)) \leq n+1$, and by Lemma~\ref{pseudosquare} we have $t(n)^2/3 \leq m(t(n))$.  Putting
these bounds together gives us
$t(n)\leq \sqrt{3(n+1)}$, which is
very close to the Brillhart-Morton upper bound for $t(n)$.

To get the other Brillhart-Morton upper bound
of Theorem~\ref{rsbounds}, we just use Eqs.~\eqref{eq6} and \eqref{eq7}, just as Brillhart and Morton did.  This gives us
$$ t(n) = s(n/2) - t(n/2 - 1) 
\leq s(n/2) \leq \sqrt{3n} $$
for $n \geq 2$ even and
$$ t(n) = s( (n-1)/2) - t( (n-1)/2 ) 
\leq s( (n-1)/2) \leq \sqrt{3(n-1)}$$
for $n$ odd.  Thus we have proved
\begin{theorem}
We have $t(n) \leq \sqrt{3n}$
for $n \geq 1$.
\end{theorem}

We can also reprove, in an extremely simple
fashion, an inequality of Brillhart and Morton on
the $\omega$ function introduced previously.
Let us start by showing 
\begin{lemma}
For $k \geq 0$ we have $\omega(k) \leq {5\over 3}m(k)$.
\end{lemma}
\begin{proof}
We verify this with the following {\tt Walnut} command.
\begin{verbatim}
eval omegabound "?msd_4 Ak,x,y ($omega(k,x) & $link42(y,k)) => 3*x<=5*y":
\end{verbatim}
and it returns {\tt TRUE}.
\end{proof}

This, combined with Lemma~\ref{pseudosquare}  gives a proof of Theorem 5 in
Brillhart and Morton \cite{Brillhart&Morton:1996}:
\begin{corollary}
We have $\omega(k) \leq {5 \over 3} k^2$ for $k \geq 0$.
\end{corollary}

\section{Counting the $k$ for which $s(k) = n$}
\label{counting}

One of the most fun properties of the Rudin-Shapiro summation function $s(n)$ is Satz 22 of \cite{Brillhart&Morton:1978}:
\begin{theorem}
There are exactly $n$ values of $k$ for which $s(k)=n$.
\label{thm22}
\end{theorem}

\begin{proof}
We can prove this theorem ``purely mechanically'' by using another capability of {\tt Walnut}:  the fact that it can create base-$b$ linear representations for values of synchronized sequences.   By a {\it base-$b$ linear representation\/} for a function $f(n)$ we mean vectors $v, w$, and a matrix-valued morphism $\gamma$ such that
$f(n) = v \gamma(x) w$ for all strings $x$ representing $n$ in base $b$.
The {\it rank\/} of a linear representation is the dimension of $v$.

So let us find a base-$2$ linear representation for the number of such $k$ for which $s(k) = n$:
\begin{verbatim}
eval satz22 n "$rss(?msd_4 k,n)":
\end{verbatim}
This gives us a base-$2$ linear representation of rank $7$ computing some function $f(n)$. 
Next we use {\tt Walnut} to compute a base-$2$ linear representation for the function $g(n) = n$:
\begin{verbatim}
eval gfunc n "i<n":
\end{verbatim}
From this, we can easily compute a base-$2$ linear representation for $f(n)-g(n)$, and minimize it using an algorithm\footnote{{\tt Maple} code implementing this algorithm is available from the second author.} of Sch\"utzenberger 
\cite[\S 2.3]{Berstel&Reutenauer:2011}.
When we do so, we get the representation for the $0$ function, so $f(n) = n$.
\end{proof}

\section{New results}
\label{new}

One big advantage to the synchronized
representation of the Rudin-Shapiro sum functions is that it becomes almost trivial to explore and rigorously prove new properties.
As a new result, let's consider the analogue of Theorem~\ref{thm22}, but for the function $t$.  Here we run into the problem that every natural number $k$ appears as a value of $t(n)$ infinitely often:
\begin{verbatim}
eval tvalues "?msd_4 An,k Em (m>n) & $rst(m,k)":
\end{verbatim}
and {\tt Walnut} returns
{\tt TRUE}.

So it makes sense to count the number of times $k$ appears as a value of $t(n)$ in some initial segment, say the first $0,1,\ldots, 2^r-1$.

Some empirical calculations suggest the following conjecture:
\begin{theorem}
\leavevmode
\begin{itemize}
\item[(a)] For $n \in [0,4^m/2)$, $0$ appears as a value of $t(n)$ exactly $2^{m-1}$ times, and $k$ appears exactly $2^m - k$ times for $1 \leq k < 2^m$.

\item[(b)] For $n \in [0,4^m)$, $0$ appears as a value of $t(n)$ exactly $2^m - 1$ times, $2^m$ appears exactly once, and $k$
appears $2(2^m -k)$ times
for $1 \leq k < 2^m$.
\end{itemize}
\end{theorem}

\begin{proof}
We use the following {\tt Walnut} commands.
\begin{verbatim}
def counta1 k x "?msd_4 $rst(n,k) & $power4(x) & x>1 & 2*n<x ":

def counta2 k x "?msd_4 Ey $power4(x) & x>1 & $link42(x,y) & 
   (?msd_2 (k=0 & 2*n<y)|(1<=k & k<y & n+k<y))":

def countb1 k x "?msd_4 $rst(n,k) & $power4(x) & n<x":

def countb2 k x "?msd_4 Ey $power4(x) & $link42(x,y) & 
   (?msd_2 (k=0 & n+1<y)|(k=y & n=0)|(1<=k & k<y & n+2*k<2*y))":
\end{verbatim}

The first two statements are used for part (a).
The code {\tt counta1} asserts that $x = 4^m$ for some $m$,  and that $t(n) =k$ for some $n<x$.  It returns a linear representation for the number of $n$ for which this holds, as a function of $k$ and $x$.  The
code {\tt counta2} creates a formula that says that the number of $n$ fulfills the conclusion of the theorem.   From these linear representations we can create a linear representation for their difference.  When we minimize it, we get the linear representation for the $0$ function, so they compute the same function.

The same approach is used
for (b).
\end{proof}

\subsection{The $\alpha$ function}

We now introduce an analogue of Brillhart and Morton's $\omega$ function,
but for the {\it first\/} occurrence of each distinct value of $s(n)$; i.e., we define
$\alpha(k)$ to be the smallest value of
$n$ for which $s(n) = k$.   We can create a $(2,4)$-synchronized 
automaton for
$\alpha$ as follows:
\begin{verbatim}
def alpha "?msd_4 $rss(n,k) & At (t<n) => ~$rss(t,k)":
\end{verbatim}
This automaton is given in Figure~\ref{alpha}.
\begin{figure}[H]
\begin{center}
\includegraphics[width=5.5in]{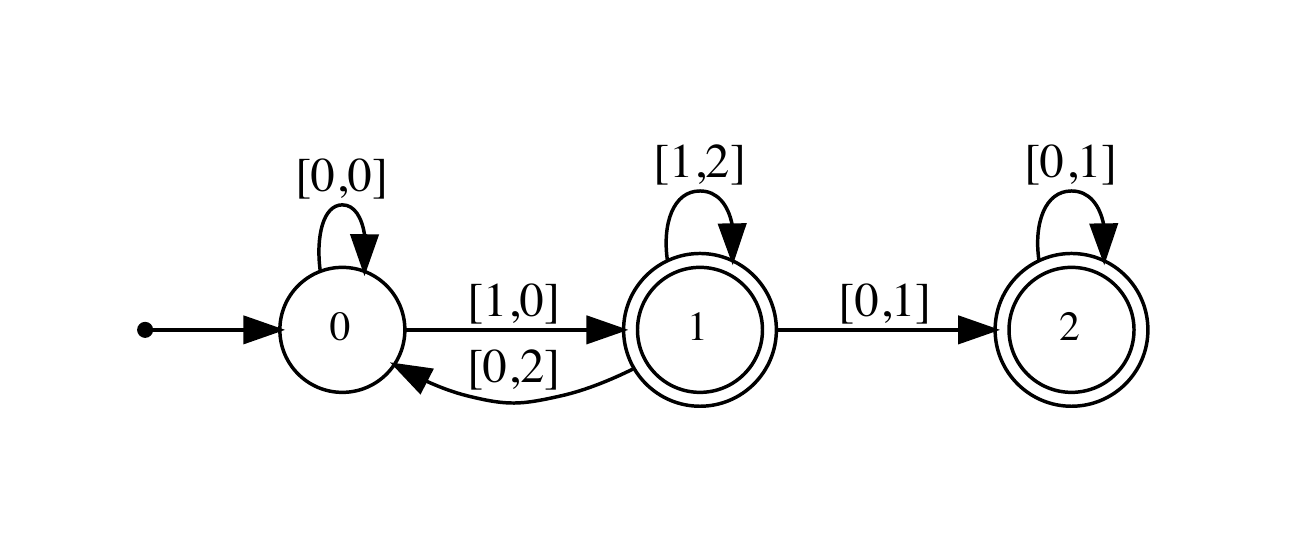}
\end{center}
\caption{$(2,4)$-synchronized automaton for $\alpha$.}
\label{alpha}
\end{figure}

\begin{theorem}
Let $k\geq 1$ and write $k=k'2^\ell$ where $k'$ is odd.  Then
$$
\alpha(k) =
m(k)/2 - (2\cdot 4^{\ell-1}+1)/3.
$$
\end{theorem}

\begin{proof}
Let $(k,n)$ be a pair accepted by the automaton in Figure~\ref{alpha}.
Suppose $\ell \geq 2$.  Note that $(k)_2$
ends with $\ell$ $0$'s and $(n)_4$ ends with $\ell$ $1$'s.  Let
$r=[2^{\ell-2}3]_4$.  We observe that $(4(n+r))_4$ consists
only of $0$'s and $2$'s and has $2$'s exactly where $(k)_2$ has
$1$'s.  It follows that $4(n+r)=2m(k)$; i.e.,
$$
4\left(n + 2\sum_{i=0}^{\ell-2}4^i + 1\right)=2m(k),
$$
which gives
$\alpha(k) = n = m(k)/2 - (2\cdot 4^{\ell-1}+1)/3$, as required.  The cases $\ell=0,1$
are similar.
\end{proof}

Similarly, we can consider the analogue of the $\alpha$ function for $t$ instead of $s$; let us call it $\alpha'$.  Remarkably, $\alpha'$ has a very simple expression in terms of known functions:
\begin{theorem}
Define $\alpha'(k) = \min \{n \suchthat t(n) = k\}$.  Then
$\alpha'(k) = m(k) -1$ for all $k \geq 1$.
\end{theorem}

\begin{proof}
We use the following {\tt Walnut} code:
\begin{verbatim}
def alphap "?msd_4 $rst(n,k) & At (t<n) => ~$rst(t,k)":
eval verify_alphap "?msd_4 Ak,t ((?msd_2 k>=1) & $alphap(k,t)) => 
   $link42(t+1,k)":
\end{verbatim}
and {\tt Walnut} returns {\tt TRUE}.
\end{proof}

\section{Plane-filling curves}
\label{curves}

In this section we show how to use our automata to prove results about the space-filling curve generated by connecting the lattice points $P_n$ in the plane
defined by $P_n = (x(n), y(n))$ for $n \geq 0$.
This curve was previously explored in the papers \cite{MendesFrance:1982,MendesFrance&Tenenbaum:1981,Dekking&MendesFrance&vanderPoorten:1982}.
The first $1024$ points of this curve are illustrated in Figure~\ref{curve}, where rounded
edges are used to make the curve clear.
\begin{figure}[H]
\begin{center}
\includegraphics[width=5in]{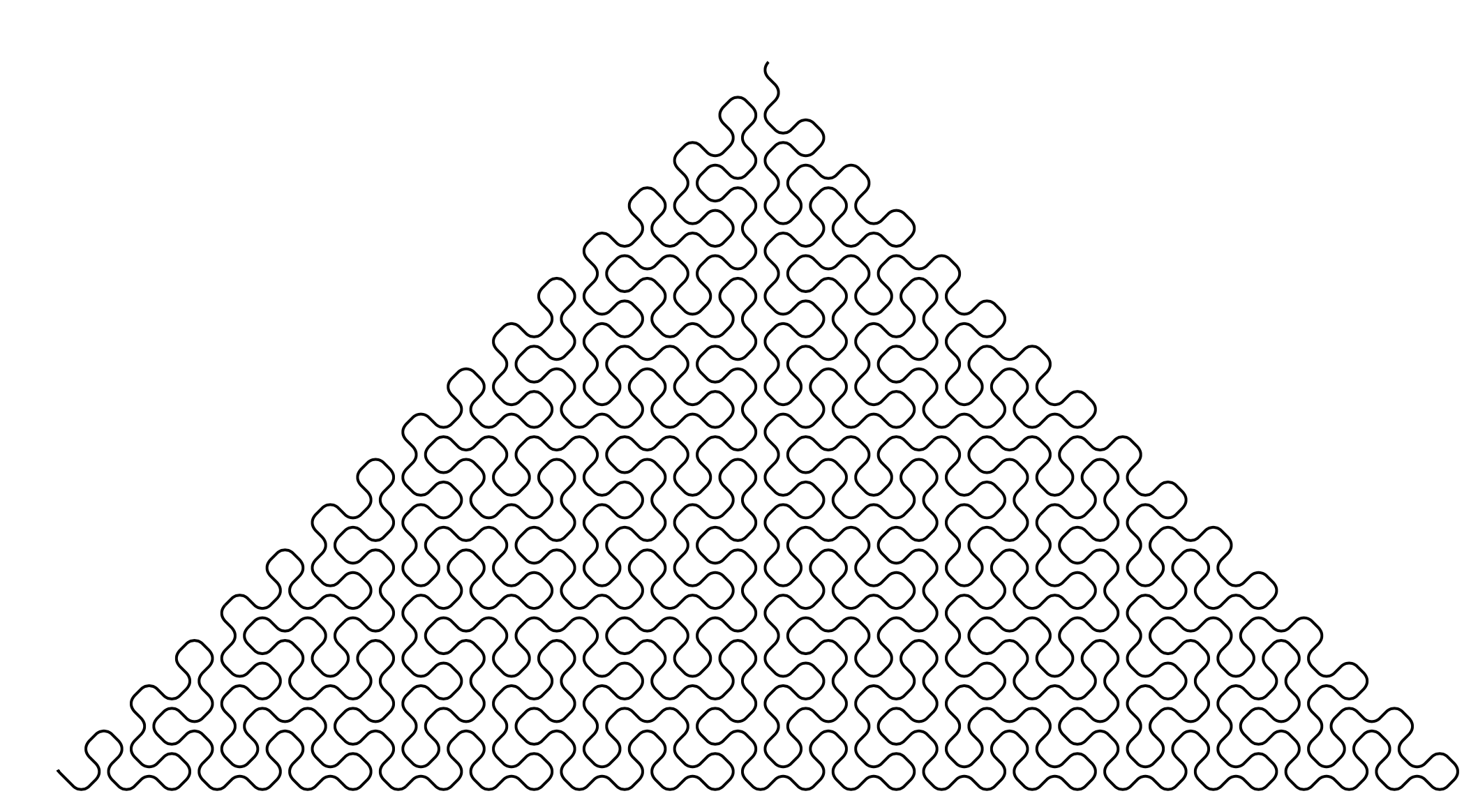}
\end{center}
\caption{The curve fills one-eighth of the plane.}
\label{curve}
\end{figure}

First let us determine exactly which lattice points are hit.   
\begin{theorem}
We have $(x,y) = (s(n),t(n))$ for $(x,y) \in \Enn\times\Enn$ if and only if
$x \geq y$ and $(x,y) \not= (0,0)$ and $x \equiv \modd{y} {2}$.  Furthermore, for each such pair $(x,y)$, there are at most two such $n$.
\end{theorem}

\begin{proof}
We use the following {\tt Walnut} code:
\begin{verbatim}
def even2 "Ek n=2*k":
def curve "?msd_4 $rss(n,x) & $rst(n,y)":
eval curvecheck "?msd_4 Ax,y (?msd_2 x>=y & x+y>0 & $even2(?msd_2 x-y)) <=> 
   En $curve(n,x,y)":
eval curvecheck3 "?msd_4 Ex,y,n1,n2,n3 n1<n2 & n2<n3 &
   $curve(n1,x,y) & $curve(n2,x,y) & $curve(n3,x,y)":
\end{verbatim}
The first check returns {\tt TRUE} and the
second, asserting a point that is hit three times, returns {\tt FALSE}.
\end{proof}

\begin{theorem}
The curve defined by $(P_n)_{n \geq 0}$ is not self-intersecting. 
\end{theorem}
\begin{proof}
    Because of the parity condition on
$x,y$ in the pairs visited it suffices to show that we never traverse the same segment
$(P_n, P_{n+1})$ twice for different $n$, either in the same direction, or the reverse direction.
We do this as follows:  we assert the existence of these traversals.
\begin{verbatim}
eval selfint1 "?msd_4 Em,x1,y1,x2,y2,n $curve(m,x1,y1) & 
   $curve(m+1,x2,y2) & $curve(n,x1,y1) & $curve(n+1,x2,y2) & m!=n":

eval selfint2 "?msd_4 Em,x1,y1,x2,y2,n $curve(m,x1,y1) & 
   $curve(m+1,x2,y2) & $curve(n+1,x1,y1) & $curve(n,x2,y2) & m!=n":
\end{verbatim}
And {\tt Walnut} returns {\tt FALSE} for both.
\end{proof}

\section{Going further}
\label{further}

Since, as mentioned in the
introduction, $a(n)$ is $+1$ or $-1$,
according to whether the number
of $11$'s occurring in $(n)_2$ are
even or odd, this suggests considering the analogous function
$a'(n)$, where we instead count the number
of $00$'s occurring in $(n)_2$.
Then it is easy to see that $a'(n)$
obeys the recursion
\begin{align*}
a'(2n) &= (-1)^{n+1} a'(n) && \quad 
(n \geq 1); \\
a'(2n+1) &= a'(n) && \quad (n \geq 0),
\end{align*}
with initial condition $a'(0) = 1$.
Then, in analogy with $s(n)$ and $t(n)$, one can consider the sums
\begin{align*}
s'(n) &= \sum_{0 \leq i \leq n} a'(n) \\
t'(n) &= \sum_{0 \leq i \leq n}
(-1)^n a'(n).
\end{align*}
The first few values of these sequences
are given in Table~\ref{tab2}.
\begin{table}[H]
\begin{center}
\begin{tabular}{c|ccccccccccccccccc}
$n$ & 0& 1& 2& 3& 4& 5& 6& 7& 8& 9&10&11&12&13&14&15\\
\hline
$a'(n)$ & 1& 1& 1& 1&$-1$& 1& 1& 1& 1&$-1$& 1& 1&$-1$& 1& 1& 1\\
$s'(n)$ & 1& 2& 3& 4& 3& 4& 5& 6& 7& 6& 7& 8& 7& 8& 9&10\\
$t'(n)$ & 1& 0& 1& 0&$-1$&$-2$&$-1$&$-2$&$-1$& 0& 1& 0&$-1$&$-2$&$-1$&$-2$ \\
\end{tabular}
\end{center}
\caption{First few values of $a'(n)$, $s'(n)$, and $t'(n)$.}
\label{tab2}
\end{table}

It turns out that both $s'(n)$
and $t'(n)$ are synchronized functions, which makes it possible to carry out the same kinds of analysis that we did
for the Rudin-Shapiro sequence.
However, since $t'(n)$ takes negative values, it's easier to work with
$1-t'(n)$ instead.
Then it is possible to prove that
both $s'(n)$ and $1-t'(n)$ are $(4,2)$-synchronized.

We just mention a few results without proof, leaving proofs and
further explorations to the reader.

\begin{theorem}
\leavevmode
\begin{itemize}
\item[(a)] We have
\begin{align*}
s'(2n) &= s'(n-1) -t'(n) + 2, && \quad (n\geq 1) \\
s'(2n+1) &= s'(n) - t'(n) + 2, &&
\quad (n \geq 0) \\
t'(2n) &= -t'(n) -s'(n-1) + 2, &&
\quad (n \geq 0) \\
t'(2n+1) &= -t'(n) -s'(n) + 2 && \quad (n \geq 0) .
\end{align*}
\item[(b)] We have
\begin{align*}
s'(4n) &= 2s'(n) - (2-(-1)^n)r'_n + 2, && \quad (n \geq 1) \\
s'(4n+1) &= 2s'(n) - 2r'_n + 2, && \quad (n \geq 0) \\
s'(4n+2) &= 2s'(n) - r'_n + 2, &&
\quad (n \geq 0) \\
s'(4n+3) &= 2s'(n) + 2, &&
\quad (n \geq 0).
\end{align*}
%\item[(c)] $\max_{n \in [2^i, 2^{i+1})} s'(n) = \begin{cases}
%2^{i/2+2} - 2, & \text{$i$ even}; \\
%3 \cdot 2^{(i+1)/2} - 2, & \text{$i$ odd}.
%\end{cases}
\item[(c)] For $k \geq 1$, the
minimum value of $s'(n)$ for
$n \in [4^k, 4^{k+1})$ is $2^{k+1} -1$ and $s'(n)$ attains this value
only when $n = (4^{k+1} -4)/3$.

For $k \geq 0$, the maximum value
of $s'(n)$ for 
$n \in [4^k, 4^{k+1})$ is
$3 \cdot 4^{k-1} -2$ and $s'(n)$
attains this value only when
$n = 4^{k+1}-1$.

\item[(d)] For $n \geq 1$ we have $3\sqrt{n}/2 \leq s'(n) \leq \sqrt{75n/7}$.
% successive maxima are at
% n = (28/3)4^i - 4/3
% s'(n) = 10\cdot 2^i - 3
% 75n-s'(n)^2 = 420\cdot 2^i - 163
% minimum achieved at n = 2

\item[(e)] $\liminf s'(n)/\sqrt{n} = \sqrt{3}$.
% successive local minima are at
% n = (4/3)2^i - 4/3
% s'(n) = 2^{i+1} - 1

\item[(f)] For $n \geq 0$ we have $
-\sqrt{24n/7} \leq t'(n) \leq 0 $.
% local maxima of t'(n)^2/n
% are at n = (14/3)*4^i - 2/3
% t'(n) = 3-2^{i+2}
\end{itemize}
\end{theorem}

Similarly, many of the results in \cite{Lafrance&Rampersad&Yee:2015} can be rederived using a $(4,2)$-synchronized automaton for their summation function $S(N)$.

\section*{Acknowledgments}

We are grateful to Jean-Paul Allouche
for several helpful suggestions.

\end{document}